\documentclass[a4paper, 11pt]{article}

\usepackage[small]{titlesec}
\usepackage{tikz-cd}
\usepackage{booktabs}
\usepackage[T1]{fontenc}
\usepackage[utf8]{inputenc}
\usepackage[english]{babel}
\usepackage[a4paper,top=4cm,bottom=3cm,left=2.5cm,right=2.5cm]{geometry}
\usepackage{amsfonts}
\usepackage{mathtools}
\usepackage{amsmath}
\usepackage{xcolor}
\usepackage{pgfplots}
\pgfplotsset{compat=1.18}
\usepgfplotslibrary{patchplots}

\usepackage{lmodern}
\usepackage{empheq}
\usepackage{amsthm}
\usepackage{amssymb}
\usepackage{booktabs}
\usepackage{caption}
\usepackage{siunitx}
\usepackage{float}
\usepackage{dsfont}
\usepackage{amssymb}
\usepackage{graphicx}
\usepackage{stmaryrd}
\usepackage{color}
\usepackage{pifont}%
\usepackage{enumerate}
\usepackage{enumitem}
\setlist[description]{leftmargin=\parindent,labelindent=\parindent}
\usepackage{listings} 
\usepackage{amsmath}
\usepackage{amscd}
\usepackage{braket}
\usepackage{physics}
\usepackage{tikz}
\usepackage{calc}
\usepackage[hidelinks]{hyperref}
\usepackage{resmes}
\usepackage{mathtools}

\newcommand{\e}{\varepsilon}

\numberwithin{equation}{section}

\usepackage{amsthm}
\newtheorem*{theoremA*}{Theorem A}
\newtheorem*{theoremB*}{Theorem B}
\newtheorem{Thm}{Theorem}[section]
\newtheorem{Prop}[Thm]{Proposition}
\newtheorem{Lem}[Thm]{Lemma}
\newtheorem{Cor}[Thm]{Corollary}

\newtheorem{Rem}[Thm]{Remark}

\newtheorem{Conj}[Thm]{Conjecture}
\newtheorem{Fact}[Thm]{Fact}

\title{Characterizing Maximal Monotone Operators \\ with Unique Representation}
\author{Sotiris ARMENIAKOS\quad\&\quad Aris DANIILIDIS}

\date{}

\begin{document}

\maketitle

\begin{abstract}
\noindent We study maximal monotone operators $A : X \rightrightarrows X^*$ whose Fitzpatrick family reduces to a singleton; such operators will be called \emph{uniquely representable}. We show that every such operator is cyclically monotone (hence, $A=\partial f$ for some convex function $f$) if and only if it is 3-monotone. In Radon--Nikod\'{y}m spaces, under mild conditions (which become superfluous in finite dimensions), we prove that a subdifferential operator $A=\partial f$ is uniquely representable if and only if $f$ is the sum of a support and an indicator function of suitable convex sets.
\end{abstract}

\smallskip

\maketitle


\vspace{0.5cm}

\textit{Secondary}: 49J53, 90C25.

\section{Introduction}
In this paper we denote by $(X,\|\cdot\|)$ a real Banach space, by $X^*$ its dual space and by $\langle \cdot, \cdot \rangle$ the duality map between $X$ and $X^*$. The graph of a set-valued operator $A: X \rightrightarrows X^*$ is denoted by $$\mathrm{Gr}(A) := \{(x,x^*)\in X\times X^*: \,x^* \in Ax\}. $$ An operator $A$ is called monotone if for every $(x,x^*), (y,y^*)\in \mathrm{Gr}(A)$, one has:
$$\langle x,x^* \rangle + \langle y,y^* \rangle \geq \langle x,y^* \rangle + \langle y,x^* \rangle.$$ A function $h: X \times X^* \to \mathbb{R}\cup \{+\infty \}$ is called a \emph{representative function} of $A$ if \\

$(R_{1})$\, $h$ is proper, convex and lower semicontinuous (in short, lsc)\,; \\

$(R_{2})$\, $h(x, x^*) \geq \langle x, x^* \rangle$, for all $(x, x^*) \in X \times X^*$\,; \\

$(R_{3})$\, $h(x, x^*) = \langle x, x^* \rangle$, for all $(x, x^*) \in \mathrm{Gr}(A)$\,. \\

We denote by $\mathcal{F}_A$ the family of all representative functions of the operator $A$ and we will refer to it as the \emph{Fitzpatrick} or the \emph{representative} family of $A$. A classical example arises when $A = \partial f$ is the subdifferential of a proper convex lsc function~$f$. In this case, the function
\begin{equation}\label{eq:Phi_f}
(x,x^*) \mapsto \Phi_{f}(x,x^*)=f(x)+f^*(x^*)
\end{equation}
belongs to the Fitzpatrick family $\mathcal{F}_{\partial f}$ (this is a consequence of the Fenchel–Young inequality), where ${f^*:X^* \rightarrow \mathbb{R}\cup \{+\infty\}}$ denotes the convex conjugate of $f$. Recall the classical fact, due to Rockafellar~\cite{Rockafellar1966}, that a maximal monotone operator is the subdifferential of a proper convex lsc function if and only if $A$ is cyclically monotone (see Section~2 for the required preliminaries). In~\cite{burachik2002maximal}, Burachik and Svaiter established a new characterization of subdifferential operators in terms of the Fitzpatrick family: $A = \partial f$ for some proper convex lsc $f$ if and only if $\mathcal{F}_A$ contains a separable representative function (that is, a function $h:X \times X^* \rightarrow \mathbb{R}$ of the form $h(x,x^*)=\phi(x)+\psi(x^*)$). 
\smallskip\newline
For maximal monotone operators, the so-called Fitzpatrick function
\begin{equation}\label{Fitzpatrick function1}
F_{A}(x,x^*) = \langle x,x^* \rangle - \inf_{(y,y^*)\in \operatorname{Gr}A} \langle x-y,x^*-y^* \rangle
\end{equation}
belongs to the Fitzpatrick family $\mathcal{F}_A$ and is in fact the pointwise minimum member of the family, see~\cite{fitzpatrick1988original}. Moreover, if two maximal monotone operators $A$, $B$ have a common representative function, then they are equal (that is, $\mathcal{F}_A\cap \mathcal{F}_B\neq \emptyset \implies A=B$). \smallskip

For merely monotone operators, a different natural representative function can be constructed as follows. We first define the function
$$(x,x^*) \mapsto \phi(x,x^*):=\,\begin{cases}
    \phantom{tri} \langle x,x^* \rangle\,,\quad \textrm{if} \,\,(x,x^*) \in \mathrm{Gr}(A) \\
    \phantom{tri}+\infty\,, \qquad\,\textrm{otherwise.}
\end{cases}$$
Then the restriction on $X \times X^*$ of the function:
\begin{equation}\label{definition P_A}
    P_{A}(x,x^*):= \phi^{**}(x,x^*)
\end{equation}
is a representative function, that is, $P_A\in \mathcal{F}_A$. In addition, $P_A$ turns out to be the pointwise maximum member of the family $\mathcal{F}_A$ (see, e.g., \cite{martinez2005monotone}). \smallskip

In particular, for a cyclically monotone operator (that is $A = \partial f$ for some proper convex, lsc function $f$) we have  
\[
\langle x, x^* \rangle \leq F_{\partial f}(x, x^*) \leq f(x) + f^*(x^*) \leq P_{\partial f}(x,x^*)\,.
\]
The above inequality reveals that the Fitzpatrick function $F_{\partial f}$ provides a more refined lower bound for the Fenchel-Young inequality. This interesting fact has been further explored in recent works initiated by Carlier,  see~\cite{carlier2023fenchel,bauschke2022carlier,burachikSvaiter2025notecarlierinequality}. \smallskip

Representation functions have been used extensively in the literature, mainly because they allow the use of convex analysis techniques in the study of monotone operators. Several applications can be found in optimization~\cite{Bot2006,Svaiter2011Application,BurachikSvaiter2003appl}, machine learning~\cite{rakotomandimby2024learningfitzpatricklosses}, optimal transport~\cite{carlier2023fenchel}, optimal control~\cite{Choussoub2008MonoFields}, stochastic differential equations~\cite{Armstrong2016RandomPDE} and Banach space theory~\cite{ReichSimons2005duality}.

\medskip

Given a maximal monotone operator $A:X \rightrightarrows X^*$, one can generally construct more than one representative function and actually infinite, as the pointwise convex combination of any two representative functions is still a representative function. The operators for which the Fitzpatrick family collapses to a singleton will be called \emph{uniquely representable}. These operators are in some sense singular. In this work, we address the following natural question:

\begin{center}
\emph{Which maximal monotone operators are uniquely representable~?}
\end{center}

This question has already been considered in the literature. For example in~\cite{BBBRW2007}, it was shown that subdifferentials of proper lower semicontinuous sublinear functions, as well as indicator functions of closed convex sets, are uniquely representable. We shall hereby extend this class of functions and obtain, under mild conditions, a complete characterization. \smallskip

In~\cite{BBW2009}, the case of linear monotone operators was studied and it was shown that such an operator is uniquely representable if and only if it is skew-symmetric. Since a linear monotone operator is a subdifferential if and only if it is symmetric, we deduce that the class of merely monotone operators contains operators which are far from being subdifferentials, yet enjoy a unique representative function. In striking contrast to the above, we shall see that imposing 3-monotonicity forces the operator to be cyclically monotone.  \smallskip

As a by-product of our analysis, we also suggest a new way to compute the Fitzpatrick function for subdifferential operators; see Lemma~\ref{Thm:F_A from f+f^*} and the subsequent remark.  \bigskip

\textit{Contributions.} \medskip 

In Section~\ref{sec-3} we prove that a maximal monotone operator which is $3-$monotone and has a unique representative function must be cyclically monotone. In other words:

\begin{theoremA*}
Let $A: X \rightrightarrows X^*$ be a maximal monotone operator such that $\mathcal{F}_A = \{F_A\}$. \\ Then the following are equivalent:\smallskip \\
{\normalfont(i).} A is 3-monotone ;\smallskip\\
{\normalfont (ii).} A is cyclically monotone. \medskip \\
More precisely, for every $v^* \in \mathrm{Im}A$, the function $x \mapsto F_{A}(x,v^*)$ is proper, convex, lsc and
\[A = \partial F_A(\cdot,v^*)\]
\end{theoremA*}

In Section~\ref{sec-4} we focus our attention to the class of subdifferential operators (that is, $A=\partial f$) and we prove the following result:

\begin{theoremB*}
Let $X$ be a space with the Radon--Nikodým property and
$f : X \to \mathbb{R}\cup \{+\infty \}$ be proper, convex, lsc with $\mathrm{int}(\mathrm{dom}\,f) \neq \varnothing$ and 
$\mathrm{int}(\mathrm{dom}\,f^{*}) \neq \varnothing$. Then the Fitzpatrick family $\mathcal{F}_{\partial f}$ consists of a single element
if and only if there exist a constant $c \in \mathbb{R}$, a functional $\overline{x}^* \in X^*$ and closed convex sets $K,C \subseteq X$ and $\mathcal{V} \subseteq X^*$ where
\begin{itemize}
    \item $C$ is a cone, $\mathcal{V}$ is $w^*$-closed convex 
 \item 
\(
    0 \in \mathcal{V} \perp K-K.
\) 
\end{itemize}
such that for every $\widehat{x} \in K$ and $x \in X$
\begin{equation} \label{eq:definition of f}
    f(x) \;=\; \sigma_{\mathcal{V}}(x - \widehat{x}) 
    \;+\; i_{\overline{K + C}}(x) 
    \;+\; \langle x, \overline{x}^* \rangle 
    \;+\; c\, .
\end{equation}
\end{theoremB*}

We prove that in finite dimensions the above technical assumptions on the domains of~$f$ and~$f^*$ can be dropped (see Theorem~\ref{thm:main4}). \medskip

A classical result of Collier~\cite{Collier} guarantees that if $X$ has the Radon--Nikodým property (in short, $X$ is an RNP space), then the dual space $X^*$ is $w^*$-Asplund, ensuring good differentiability properties for the conjugate functions $f^*$. This plays a central role in our approach. It is quite interesting to investigate whether the class of admissible spaces can be extended to classes of spaces where convex functions satisfy weaker (if any) differentiability properties. \smallskip

Finally, without assumptions on $X$ we generalize \cite[Theorem~5.3]{BBBRW2007}; see Theorem~\ref{Thm:Main2} and the subsequent remarks.

\medskip

\noindent\textit{Organization of the paper.} \ The paper is structured as follows. In Section~2, we recall fundamental results from convex analysis and differentiability theory in Banach spaces. Section~3 is devoted to the proof of our first main result (Theorem~A) concerning $3$-monotone operators and in Section~4 we focus on the particular case of subdifferential operators and prove several results leading eventually to Theorem~B. We also discuss the finite dimensional case.

\section{Preliminaries}

\phantom{Tri} In this section we fix our notation and recall notions and preliminary results that will be used throughout the paper. \bigskip

\noindent\textbf{Convex functions and subdifferentials.} Let $f \colon X \rightarrow \mathbb{R}\cup \{+\infty \}$ be a proper convex function. A functional $x^* \in X^*$ is called an $\e$-\emph{subgradient} of $f$ at $x$ if
\[
f(y) - f(x) \;\geq\; \langle y - x, x^* \rangle - \e , \quad \text{for every} \quad y \in \mathrm{dom}\, f\,.
\]
For $\e \geq 0$, the set of all $\e$-subgradients of $f$ at $x$ is denoted by $\partial_{\e} f(x)$ and called the $\e$-subdifferential of $f$ at $x$. In the particular case that $\e=0$, $\partial_{0}f(x)$ is denoted simply by $\partial f(x)$ and called the subdifferential of $f$ at $x$. The Fenchel–Legendre conjugate of $f$ is the proper convex $w^*$-lsc function $f^* \colon X^* \to \mathbb{R}\cup \{+\infty \}$ defined by
\[
f^*(x^*) := \sup_{x \in X} \bigl\{ \langle x, x^* \rangle - f(x) \bigr\}.
\]
For every $(x,x^*) \in X \times X^*$,  the Fenchel–Young inequality holds:
\[
f(x)+f^*(x^*) \;\geq\; \langle x,x^* \rangle\,,
\]
and $\e$-subgradients of $f$ are characterized via $f^*$ as follows
\[
x^* \in \partial_{\e}f(x) 
\;\;\iff\;\; f(x)+f^*(x^*) \;\leq\; \langle x,x^* \rangle + \e.
\]

\noindent It is well known that if $f \colon X \rightarrow \mathbb{R}\cup \{+\infty \}$ is proper, convex, lsc then for any $x \in \mathrm{int}(\mathrm{dom}\,f)$ and $\e \geq 0$, the set of $\e$-subdifferentials is bounded. More precisely, there exists $M = M(x,f,\e)$ such that
\[
\|x^*\| \leq M, \quad \text{for all} \,\, x^* \in \partial_{\e}f(x).
\]

\medskip

\noindent For a subset $A \subseteq X$, the indicator function $i_{A}$ is defined as
\[
i_A(x) := \begin{cases}
\,\,\,\,\,\,\,0\,, & \text{if } x \in A, \\
\,\,+\infty\,, & \text{otherwise}.
\end{cases}
\]
For a set $B \subseteq X^*$, the support function $\sigma_{B}$ is given by
\[
\sigma_{B}(x) := \sup_{x^* \in B} \langle x, x^* \rangle.
\]
Given a set $C \subseteq X^*$, the polar set $C^{\circ} \subseteq X$ is defined by
\[
C^{\circ} := \{ x \in X : \langle x,x^* \rangle \leq 1, \,\, \forall x^* \in C\}.
\]
In the special case where $C$ is a cone, one has
\[
\sigma_{C}(x) = i_{C^{\circ}}(x).
\]
\bigskip

\noindent\textbf{Inf-convolution.} Given two functions $f, g \colon E \to (-\infty, +\infty]$, where $E$ is a real vector space, their \emph{infimal convolution} is defined by
\[
(f \,\square\, g)(x) := \inf_{y \in E} \bigl\{ f(y) + g(x - y) \bigr\}, \quad x \in E.
\]
For a function $f\colon X \rightarrow \mathbb{R}\cup \{+\infty \}$, we define its lower semicontinuous envelope $\overline{f}$, as follows
\begin{equation}\label{eq:lsc envelope}
\overline{f} := \sup\{ g : g \;\text{is lsc and } g \leq f\},
\end{equation}
that is, the greatest lower semicontinuous function dominated by $f$. We also recall the following identity:
\begin{equation}\label{eq:duality infimal conv}
(\overline{f \square g})^* = f^*+g^*,
\end{equation}
for any two proper functions $f,g$.
\bigskip

\noindent\textbf{Differentiability and RN spaces.} A function $\phi \colon X \rightarrow Y$ is said to be \emph{Fréchet differentiable} at $x$ if there exists a linear bounded operator $A \colon X \to Y$ such that
\[
\lim_{\|u\|_X \to 0} \frac{\|\phi(x + u) - \phi(x) - A(u)\|_Y}{\|u\|_X} = 0.
\]
In this case, $A$ is unique and denoted $D\phi(x)$ (the Fr\'{e}chet derivative of $f$ at $x$.). In particular, for a proper convex $f \colon X \to \mathbb{R}\cup \{+\infty \}$, a point $x$ is called a point of Fréchet differentiability (or a \emph{Fréchet point}) if the subdifferential mapping $\partial f$ is single-valued and norm-to-norm upper semicontinuous at $x$. In this case, the Fréchet derivative agrees with the unique subgradient. \smallskip

Convex functions exhibit strong regularity properties. In the finite-dimensional case, the Rademacher theorem ensures that every locally Lipschitz function (in particular every convex function) is differentiable almost everywhere. In the infinite-dimensional setting, convex functions retain partial differentiability properties on a large set if the underlying space is an \emph{Asplund space}. A Banach space $X$ is called Asplund if every convex continuous function defined on an open convex subset of $X$ is Fréchet differentiable on a dense $G_\delta$ subset of its domain. Notably, reflexive spaces and Banach spaces with separable duals are Asplund. \smallskip

Let us further recall that a \emph{slice} $S(x^*,A,a)$ in a Banach space $X$ is defined for $A\subset X$, $a>0$ and $x^* \in X^*$ by
\[
S(x^*,A,a)=  \{x \in A : \langle x,x^* \rangle > \sigma_{A}(x^*)-a \}.
\]
A nonempty set $A \subseteq X$ is called \emph{dentable} if it admits slices of arbitrarily small diameter. A Banach space $X$ is said to have the RNP (or called an RN space) if every nonempty bounded subset of $A$ is dentable. One of the central results in the differentiability theory of convex functions is the fact that
\[
X \;\;\text{is Asplund} \;\;\iff\;\; X^* \;\;\text{has the RNP}.
\]
In this work we are going to use a dual version of the above due to Collier~\cite{Collier}
\[
X \;\;\text{has the RNP} \;\;\iff\;\; X^* \;\;\text{is $w^*$-Asplund}.
\]
A dual space $X^*$ is called $w^*$-Asplund if every convex, continuous and $w^*$-lsc function on $X^*$ is Fréchet differentiable on a $G_\delta$ dense subset of its domain. We remark further that if $\widehat{a}=Df^*(v^*)$ for a Fr\'{e}chet point $v^*$ of $f^*$, then $\widehat{a}\in X$. For further background on convex analysis and these functional-analytic preliminaries, see~\cite{Fabianbook,Phelps1993,Zalinescubook}.
\bigskip

\noindent\textbf{Monotone operators and representative functions.} For an operator $A \colon X \rightrightarrows X^*$, we denote the domain by $\mathrm{dom}\,(A) = \{ x \in X : Ax \neq \varnothing \}$ and the range by $$\mathrm{Im}(A) = \{x^* \in X^* : \exists x \in X \;\text{s.t. } x^* \in Ax\}.$$ An operator $A \colon X \rightrightarrows X^*$ is said to be \emph{$n$-cyclically monotone} (or simply $n$-monotone) if for any collection $\{(a_i, a_i^*)\}_{i=1}^n \subseteq \mathrm{Gr}(A)$ we have
\[
\sum_{i=1}^{n} \langle a_i, a_i^* \rangle \;\;\geq\;\; \sum_{i=1}^{n} \langle a_{i+1},a^*_{i} \rangle,
\]
with the convention $a_{n+1} = a_1$. The operator $A$ is called cyclically monotone if it is $n$-monotone for every $n \in \mathbb{N}$. It is called \emph{monotone} if it is 2-monotone, which is equivalent to the condition
\[
\langle x - y, x^* - y^* \rangle \geq 0, \quad \text{for every}\quad (x, x^*), (y, y^*) \in \mathrm{Gr}(A).
\]
A monotone (resp. cyclically monotone) operator is called maximal monotone (resp. maximal cyclically monotone) if there is no strict extension that preserves monotonicity (resp. cyclic monotonicity). It is well-known that subdifferentials of proper convex lsc functions are maximal cyclically monotone operators. \smallskip

The \emph{Fitzpatrick function} associated to an operator $A$ is the function $F_A \colon X \times X^* \to \mathbb{R}\cup \{+\infty \}$ defined by \eqref{Fitzpatrick function1}, or equivalently by
\begin{equation}\label{eq:Definition of Fitzpatrick2}
F_A(x, x^*) := \sup_{(y, y^*) \in \mathrm{Gr}(A)} \bigl\{ \langle y, x^* \rangle + \langle x,y^*  \rangle - \langle y, y^* \rangle \bigr\}.
\end{equation}
Recall also the function $P_{A}$, defined by \eqref{definition P_A}. The following properties are folklore, see e.g.~\cite{fitzpatrick1988original,martinez2005monotone}. If $A \colon X \rightrightarrows X^*$ is maximal monotone, then $F_A \in \mathcal{F}_A$, and $F_A$ is the pointwise minimum of~$\mathcal{F}_A$, satisfying
\[
F_A(x, x^*) = \langle x, x^* \rangle \quad \iff \quad (x, x^*) \in \mathrm{Gr}(A).
\]
For any monotone operator $A$, we have $P_{A} \in \mathcal{F}_{A}$.
Moreover, $P_A$ is the pointwise maximal member of~$\mathcal{F}_A$. In addition, for $x \in X$ and $x^* \in X^*$ we have:
\begin{equation}\label{eq:F_A and P_A}
P_{A}(x,x^*)=F^*_{A}(x^*,x) = \sup_{(y,y^*)\in X\times X^*} \left\{ \langle y,x^* \rangle + \langle x,y^*\rangle - F_{A}(y,y^*)\right\}.
\end{equation}
In addition,  $P_{A}$ satisfies the following inequality for all $(x,x^*), (y,y^*) \in X \times X^*$ (see ~\cite[Proposition~2.7]{Borwein2006})
\begin{equation}\label{eq:P_{A} main}
P_A(x,x^*) + P_A(y,y^*) \;\;\geq\;\; \langle x, y^* \rangle + \langle y, x^* \rangle.
\end{equation}
\medskip

\noindent\textbf{Two useful facts.} We finish this section by mentioning two auxiliary results that will be used in the sequel. The first one follows from~\cite[Proposition~3.6]{BD2002} and~\cite[Corollary~3.7]{BD2002}. 
\begin{Fact}\label{BD2002}
    Let $f,g:X \rightarrow \mathbb{R} \cup \{+\infty \}$ be proper, convex and lsc such that for all $x \in X$:
    \[f(x) \geq g(x)\]
Then: \\
{\normalfont (i)}. If $X$ is finite dimensional and $f(x)=g(x)$ for a dense subset of $\mathrm{rint}(\mathrm{dom} \,g)$ then:
\[f(x)=g(x) \quad \text{for all} \quad x \in X.\]
{\normalfont (ii)}. If $\mathrm{int}(\mathrm{dom} \, g) \neq \varnothing$ and $f(x)=g(x)$ for a dense subset of $\mathrm{int}(\mathrm{dom}\,g)$, then
\[f(x)=g(x) \quad \text{for all} \quad x \in X.\]
\end{Fact}
\noindent We mention for completeness that the assumption $\mathrm{int}(\mathrm{dom}\,g) \neq\varnothing$ is essential in~(ii) above, see~\cite[Proposition~3.4]{BD2002}. \smallskip\newline 
The second fact is a structural result for subgradients of a convex function. The result can be deduced from~\cite[Corollary~4.10]{BorweinYao2013}.
\begin{Fact}\label{BorweinYao2013}
    Let $f : X \rightarrow \mathbb{R} \cup \{+\infty\}$ be proper, convex, lsc and $\mathrm{int}(\mathrm{dom}\,f) \neq \varnothing$. Then for every $x \in \mathrm{dom}\,\partial f$ and $x^* \in \partial f(x)$, there exists $x_{\mathrm{int}}^* \in \overline{\mathrm{conv}}^{w*}(\bigcup_{y \in \mathrm{int}(\mathrm{dom}\, f)} \partial f(y))$ and $x_{N}^* \in N_{\mathrm{dom}\,f}(x)$ such that:
    \[x^* \,=\, x^*_{\mathrm{int}}\,+\,x^*_{\mathrm{N}}\,.\]
    
\end{Fact}

\section{Case of Maximal Monotone Operators}\label{sec-3}
The aim of this section is to establish Theorem~A, asserting that every maximal monotone, $3$-monotone  operator with $\mathcal{F}_A$ singleton is in fact a subdifferential. To this end, let us first present some general properties that a maximal monotone operator $A$ should necessarily satisfy if the Fitzpatrick family $\mathcal{F}_A$ is a singleton. We begin by the following observation. 

\begin{Lem}\label{lem:singleton characterization}
    For a maximal monotone operator $A:X \rightrightarrows X^*$ the following are equivalent:\smallskip
    
    {\normalfont(i).} $\mathcal{F}_A$ is a singleton, i.e. $\mathcal{F}_A=\{F_{A}\}$.\smallskip
    
    {\normalfont(ii).} $F_{A} \equiv P_{A}$. \smallskip
    
    {\normalfont(iii).} For every $(x,x^*),(y,y^*) \in X \times X^*$
    \begin{equation}\label{eq:F_A main}
F_A(x,x^*) + F_A(y,y^*) \;\geq\; \langle x, y^* \rangle + \langle y, x^* \rangle.
\end{equation}
If $A = \partial f$ for some proper convex lsc function $f:X \rightarrow \mathbb{R} \cup \{+\infty\}$, then the above assertions are also equivalent to the following property, evoking the function $\Phi_{f}$ from \eqref{eq:Phi_f}: \smallskip

{\normalfont(iv).} For every $(x,x^*) \in X \times X^*$
\[F_{\partial f}(x,x^*)=\Phi_{f}(x,x^*). \]
\end{Lem}
\begin{proof} Since $F_A$ is the minimal and $P_A$ is the maximal member of $\mathcal{F}_A$, the equivalence between~(i) and~(ii) is straightforward. The implication (ii) \!$\implies$\! (iii) follows from \eqref{eq:P_{A} main}. Furthermore \eqref{eq:F_A main} together with \eqref{eq:F_A and P_A} implies
 \[F_{A}(x,x^*) \geq \sup_{(y,y^*)\in X\times X^*} \left\{ \langle y,x^* \rangle + \langle x,y^*\rangle - F_{A}(y,y^*)\right\} = P_{A}(x,x^*)\]
in virtue of \eqref{eq:F_A and P_A}. Since we always have $F_{A} \leq P_{A}$, we deduce that (ii) holds.

\medskip 
Let us now assume that $A= \partial f$ for a proper convex lsc function $f:X \rightarrow \mathbb{R}\cup \{+\infty\}$. Since both functions $\Phi_{f},F_{\partial f}$  belong to the family $\mathcal{F}_{\partial f}$ we clearly have that~(i) \!$\implies$\! (iv). We shall show that (iv) \!$\implies$\! (iii). Indeed, if (iv) holds, then in virtue of the Fenchel--Young inequality we obtain:
\[F_{\partial f}(x,x^*)+F_{\partial f}(y,y^*) = f(x)+f^*(y^*)+f(y)+f^*(x^*) \geq \langle x,y^* \rangle +\langle y,x^* \rangle \]
concluding the proof. 
\end{proof}

We proceed by observing that the class of operators for which $\mathcal{F}_A = \{F_A\}$ is invariant under both translations and dilations of the graph.

\begin{Prop}[Translation and Dilation Invariance]\label{prop:Translation Invariance}
Let $A : X \rightrightarrows X^*$ be a maximal monotone operator such that $\mathcal{F}_A = \{F_A\}$. 
Then for every $\lambda_1, \lambda_2 > 0$ and $(w,w^*) \in X \times X^*$, 
the operator $\widehat{A}$ with graph
\[
\mathrm{Gr}(\widehat{A}) = \bigl\{\, (\lambda_1 x - w, \, \lambda_2 x^* - w^*) 
: (x,x^*) \in \mathrm{Gr}(A) \,\bigr\}
\]
is also maximal monotone and satisfies $\mathcal{F}_{\widehat{A}} = \{F_{\widehat{A}}\}$.
\end{Prop}

\begin{proof}
By the previous lemma, we have
\begin{equation*}
F_A(x,x^*) + F_A(y,y^*) \;\geq\; \langle x, y^* \rangle + \langle y, x^* \rangle,
\quad \text{for every} \,\,\, (x,x^*),(y,y^*) \in X \times X^*.
\end{equation*}
Let us fix $(w,w^*) \in X \times X^*$ and define $A_1$ by
\[
\mathrm{Gr}(A_1) := \mathrm{Gr}(A) - \{(w,w^*)\}.
\]
A direct calculation shows that
\begin{align*}
&F_{A_1}(x - w, x^* - w^*) + F_{A_1}(y - w, y^* - w^*) \\
&\quad = F_A(x, x^*) + F_A(y, y^*) - \langle x+y, w^* \rangle - \langle w, x^*+y^* \rangle + 2\langle w,w^* \rangle.
\end{align*}
Using \eqref{eq:F_A main}, the above expression is bounded from below by
\[
\langle x-w, y^*-w^* \rangle + \langle y-w, x^*-w^* \rangle,
\]
so in view of Lemma~\ref{lem:singleton characterization}, in particular from (iii) $\implies $ (i), we have that $\mathcal{F}_{A_1} =\{F_{A_1}\}$. Fix now $\lambda_1,\lambda_2 > 0$ and define
\[
\mathrm{Gr}(A_2) := \{ (\lambda_1 a, \lambda_2 a^*) : (a,a^*) \in \mathrm{Gr}(A) \}.
\]
Then
\[
F_{A_2}(\lambda_1 x, \lambda_2 x^*) = \lambda_1 \lambda_2 F_A(x,x^*).
\]
Therefore, if the operator~$A$ satisfies the inequality~\eqref{eq:F_A main}, so does the operator~$A_{2}$. \smallskip

The assertion follows by combining the above cases.
\end{proof}

\medskip

\noindent We now introduce marginal functions associated to representative functions.

\begin{Lem} Let $A:X \rightrightarrows X^*$ be a monotone operator. For every $w \in \mathrm{dom}\,(A)$ and $v^* \in \mathrm{Im}(A)$ the marginal functions $f:X \rightarrow \mathbb{R}\cup \{+\infty \}$ and $g:X^* \rightarrow \mathbb{R}\cup \{+\infty \}$ defined through:
\[
f(x) = f_{A,w}(x) := \inf_{a^* \in X^*} \left\{ P_A(x, a^*) - \langle w, a^* \rangle \right\}, \]
and
\[g(x^*)= g_{A,v^*}(x^*) := \inf_{a \in X} \left\{ P_A(a, x^*) - \langle a, v^* \rangle \right\}.
\]
are proper convex and lsc.
\end{Lem} 
\begin{proof} We first prove that $f$ is proper. Indeed, since $A(w) \neq \varnothing$, we can consider $w^* \in A(w)$, then since $P_A$ is a representative function we have 
\[P_{A}(w,w^*)=\langle w,w^* \rangle \]
and we deduce from~\eqref{eq:P_{A} main} that
\begin{align*}
P_{A}(x,a^*)-\langle w,a^* \rangle &= P_{A}(x,a^*)+\underbrace{P_{A}(w,w^*)-\langle w,w^* \rangle}_{=0}- \langle w,a^* \rangle \\
&\geq \langle x,w^* \rangle -\langle w,w^* \rangle. 
\end{align*}
Consequently $f(x) > -\infty$ for all $x \in X$. Moreover, $f$ is convex and lsc, as a marginal function of the (jointly) convex function:
\[(x,a^*) \mapsto P_{A}(x,a^*)-\langle w,a^* \rangle  \]
The assertion for $g$ follows similarly. \end{proof}

The above functions are quite useful in the study of maximal monotone operators. A variant appears in~\cite{Borwein2006} to prove a central case of the Debrunner–Flor theorem. More importantly in our work, the duals of these functions interpolate between $F_{A}(\cdot,v^*)$ and $P_{A}(\cdot,v^*)$:

\begin{Lem}\label{lem:marginals}
Let $A$ be a maximal monotone operator. For every $w \in \mathrm{dom}\,(A)$, $v^* \in Im(A)$ and $(x,x^*)\in X \times X^*$ we have: \medskip

{\normalfont(i).} \ $F_A(x, v^*) \leq g_{A,v^*}^*(x) \leq P_A(x, v^*)\,;$ \medskip

{\normalfont(ii).} \ $F_A(w, x^*) \leq f_{A,w}^*(x^*) \leq P_A(w, x^*)$\,.
\end{Lem}

\begin{proof}
(i). From \eqref{eq:P_{A} main}, for all $a \in X$ and $x^* \in X^*$,
\begin{align*}
P_A(x, v^*) + P_A(a, x^*) \geq \langle x, x^* \rangle + \langle a, v^* \rangle \implies \\[4pt]
P_{A}(x,v^*)+\underbrace{\inf_{a \in X}\left\{ P_{A}(a,x^*)-\langle a,v^* \rangle \right\}}_{g_{A,v^*}} \geq \langle x,x^* \rangle \implies \\
P_A(x, v^*) \geq \sup_{x^* \in X^*} \left( \langle x, x^* \rangle - g_{A,v^*}(x^*) \right) = g_{A,v^*}^*(x).
\end{align*}
On the other hand,
\begin{align*}
F_A(x, v^*) &= \sup_{(a,a^*) \in \mathrm{Gr}(A)} \left\{ \langle x, a^* \rangle + \langle a, v^* \rangle - \langle a, a^* \rangle \right\} \\
&= \sup_{(a,a^*) \in \mathrm{Gr}(A)} \left\{ \langle x, a^* \rangle + \langle a, v^* \rangle - P_A(a, a^*) \right\} \\
&\leq \sup_{(a,a^*)\in \mathrm{Gr}(A)} \left\{ \langle x, a^* \rangle + \sup_{a\in X}\left( \langle a, v^* \rangle - P_A(a, a^*) \right)\right\}\\
&=\sup_{(a,a^*) \in \mathrm{Gr}(A)} \left\{ \langle x,a^*\rangle - \inf_{a\in X}\left( P_{A}(a,a^*)-\langle a,v^* \rangle \right) \right\}\\
&= \sup_{(a,a^*) \in \mathrm{Gr}(A)} \left\{ \langle x, a^* \rangle - g_{A,v^*}(a^*) \right\} \leq \sup_{a^* \in X^*} \left\{ \langle x, a^* \rangle - g_{A,v^*}(a^*) \right\}=g_{A,v^*}^*(x).\\
\end{align*}
(ii). It follows analogously.
\end{proof}
As an immediate corollary we deduce the following.

\begin{Cor}
Let $A: X \rightrightarrows X^*$ be a maximal monotone operator such that $\mathcal{F}_A = \{F_A\}$. Then for every $w \in \mathrm{dom}\,(A)$ and $v^* \in \mathrm{Im}(A)$ we have: \medskip

{\normalfont(i).} \ $F_A(x, v^*) = g_{A,v^*}^*(x) = P_A(x, v^*) \,;$ \medskip

{\normalfont(ii).} \ $F_A(w, x^*) = f_{A,w}^*(x^*) = P_A(w, x^*)\,.$
\end{Cor}

\begin{proof}
It follows directly from Lemma~\ref{lem:marginals} and Lemma~\ref{lem:singleton characterization} (ii).
\end{proof}

With this in hand, we are ready to prove one of the two main results of this work. \smallskip \newline
\noindent\rule{4cm}{1.6pt}\smallskip \newline
\noindent\textit{Proof of Theorem~A.} \smallskip\newline 
(i)\!$\implies$\!(ii). Let $(x, x^*) \in \mathrm{Gr}(A)$ and $v^* \in \mathrm{Im}(A)$ be arbitrary. By the Fenchel-Young inequality and the previous corollary we get
\begin{equation}\label{eq:3.6.1}
F_{A}(x,v^*)+g_{A, v^*}(x^*)=g^*_{A, v^*}(x)+ g_{A,v^*}(x^*) \geq \langle x, x^* \rangle.
\end{equation}
We need to show that equality holds, which directly yields that $x^* \in \partial g^*_{A,v^*}(x)= \partial F_{A}(\cdot,v^*)$. Arguing by contradiction, assume that the inequality is strict, that is,
\[
F_A(x, v^*) + g_{A, v^*}(x^*) > \langle x, x^* \rangle.
\]
Then, there would exist $(y, y^*) \in \mathrm{Gr}(A)$ such that
\[
\langle x, y^* \rangle + \langle y, v^* \rangle + g_{A, v^*}(x^*) > \langle x, x^* \rangle + \langle y, y^* \rangle.
\]
But since $g_{A, v^*}(x^*) \leq P_A(y, x^*) - \langle y, v^* \rangle = F_A(y, x^*) - \langle y, v^* \rangle $, it would follow that
\[
\langle x, y^* \rangle + F_A(y, x^*) > \langle x, x^* \rangle + \langle y, y^* \rangle.
\]
Thus there would exist $ (z, z^*) \in \mathrm{Gr}(A)$ such that
\[
\langle x, y^* \rangle + \langle y, z^* \rangle + \langle z, x^* \rangle > \langle x, x^* \rangle + \langle y, y^* \rangle + \langle z, z^* \rangle,
\]
contradicting the $3$-monotonicity of $A$. Therefore in \eqref{eq:3.6.1} equality should hold, namely,
\[
F_A(x, v^*) + g_{A, v^*}(x^*) = \langle x, x^* \rangle,
\]
which implies \( x^* \in \partial F_A( \cdot, v^*)(x) \). Since the pair $(x, x^*)$ was arbitrary in $\mathrm{Gr}(A)$ and both operators $A$ and $\partial F_A(\cdot, v^*)$ are maximal monotone, we conclude
\[
A = \partial F_A(\cdot, v^*).
\]
\noindent (ii)\!$\implies$\!(i). If $A$ is cyclically monotone, then clearly $A$ is 3-monotone. For the "more precisely part", in virtue of the Rockafellar characterization theorem, $A= \partial f$ for a proper convex lsc function $f$. Fix now $v^* \in \mathrm{Im}(A)$, so that $f^*(v^*) < + \infty$. Since $\mathcal{F}_A$ is a singleton, by Lemma~\ref{lem:singleton characterization}(iv) we have
\[\Phi_{f}(x,v^*)=f(x)+f^*(v^*)= F_{A}(x,v^*), \quad \text{ for all} \,\,\,x\in \mathrm{dom}\,f.\]
Since $v^*$ is fixed
\[A(x)=\partial f(x)=\partial (f(x)+f^*(v^*))=\partial F_{A}(\cdot ,v^*)(x), \]
that is,
\[A=\partial F_{A}(\cdot,v^*)\]
for all $v^* \in \mathrm{Im}(A)$, which concludes the proof.\hfill$\square$

\begin{Rem}\normalfont\label{Rem:3.7}
(i). It was proven in~\cite{BBW2009} that a linear monotone operator~$A$ is skew-symmetric if and only if $\mathcal{F}_A$ is a singleton. As a consequence, $3$-monotonicity is a crucial assumption in the statement of Theorem~A.\smallskip

\noindent(ii). The notion of \textit{paramonotone} (respectively, $3^*$-\textit{monotone}) operator interpolates between between the classes of monotone and $3$-monotone operators, see~\cite[Chapter~22]{BauschkeCombetes} and~\cite{BrezisArticle1973} for relevant definitions and an exposition. It is reasonable to conjecture that paramonotone (resp. $3^*-$monotone) maximal monotone operators with $\mathcal{F}_A$ singleton are necessarily subdifferentials. Unfortunately, our current techniques do not provide an answer to this question.
\end{Rem}

\section{Case of Subdifferential Operators} \label{sec-4}

\subsection{Structural Properties of Subdifferentials with Unique Representative Function}

\noindent By virtue of the previous section, we now focus on subdifferential operators $A=\partial f$ where function $f:X \rightarrow \mathbb{R} \cup \{+\infty\}$ is proper convex lsc . As mentioned in the introduction, the function $\Phi_{f}:X \times X^* \rightarrow \mathbb{R}\cup \{+\infty \}$ defined by:
\[(x,x^*) \mapsto \Phi_{f}(x,x^*) = f(x)+f^*(x^*)\]
belongs to $\mathcal{F}_{\partial f}$ and consequently $\Phi_{f} \geq F_{\partial f}$ on $X \times X^*$. For any monotone operator, we define the following set-valued mapping $\mathcal{M}_{A}: X \times X^* \rightrightarrows \mathrm{Gr}(A)$,
\begin{equation}\label{eq:definition of M_A}
\mathcal{M}_{A}(x,x^*)= \{ (a,a^*) \in \mathrm{Gr}(A) \,\, : F_{A}(x,x^*)=\langle x,a^* \rangle +\langle a,x^* \rangle -\langle a,a^* \rangle \}
\end{equation}
that is, the set of all points of $\mathrm{Gr}(A)$ realizing the supremum in \eqref{eq:Definition of Fitzpatrick2}. In our case of interest we have the following result. 

\begin{Prop}\label{Prop:Spicification of M_A}
Let $f$ be proper convex lsc and assume that $F_{\partial f}(x,x^*)=\Phi_{f}(x,x^*)$ for some
$(x,x^*)\in \mathrm{dom}\, \Phi_{f}$. Then for every $(a,a^*)\in \mathcal{M}_{\partial f}(x,x^*)$ we have:
\[a^* \in \partial f(x) \cap \partial f(a) \qquad \text{and} \qquad a \in \partial f^*(x^*) \cap \partial f^*(a^*) \]
\end{Prop}
\begin{proof}
    in view of \eqref{eq:definition of M_A} for every $(a,a^*)\in \mathcal{M}_{\partial f}(x,x^*)$ we have 
    \[
    F_{\partial f}(x,x^*)+\langle a,a^* \rangle \, =\,  \langle a,x^* \rangle +\langle x,a^* \rangle\,.
    \]
    Since $F_{\partial f}(x,x^*)=\Phi_{f}(x,x^*)=f(x)+f^*(x^*)$ and $f(a)+f^*(a^*)=\langle a,a^*\rangle $, we deduce:
    \[
    f(x)+f^*(x^*)+f(a)+f^*(a^*)= \langle a,x^* \rangle + \langle x,a^* \rangle\,.
    \]  
Using the Fenchel--Young inequality we infer from the above that
    \[f(x)+f^*(a^*)= \langle x,a^*\rangle \iff a^* \in \partial f(x)\]
    and
    \[f(a)+f^*(x^*)=\langle a,x^* \rangle \iff a \in \partial f^*(x^*)\,.\]
    Since $a^* \in \partial f(a)$, the conclusion follows. 
\end{proof}

The required technical tool to proceed is the following lemma.

\begin{Lem}\label{Lem:Fitzpatrick maximizer}
Let $f \colon X \to \mathbb{R}\cup \{+\infty \}$ be proper, convex and lsc. Let $x \in \mathrm{int}(\mathrm{dom}\, f$), $v^* \in \mathrm{dom}\,(Df^*)$ and define $\widehat{a}:= Df^*(v^*)$. Assume that $F_{\partial f}(x,v^*)=\Phi_{f}(x,v^*)$. Then: \smallskip

{\normalfont(i).} $\mathcal{M}_{\partial f}(x,v^*) \neq \varnothing$
\smallskip

{\normalfont(ii).} For every $(a,a^*) \in \mathcal{M}_{\partial f}(x,v^*)$ we have: 
\[
 a= \widehat{a}=Df^*(v^*) \quad\text{and }\quad a^* \in \partial f(x) \cap \partial f(\widehat{a}).
\]

\end{Lem}

\begin{proof}
(i). Since $(x,v^*) \in \mathrm{dom}\, f \times \mathrm{dom}\, f^* =\mathrm{dom}\,\Phi_{f}$ we have
\[
F_{\partial f}(x,v^*)= \Phi_{f}(x,v^*) < +\infty.
\]
Consider a positive sequence $\{\e_{n}\}_{n \in \mathbb{N}}$ with $\e _{n} \rightarrow 0$ and pick $\{a_{n},a^*_{n}\}_{n \in \mathbb{N}} \subseteq \mathrm{Gr}(\partial f)$ such that
\[
F_{\partial f}(x,v^*) \leq \langle x, a_n^* \rangle + \langle a_{n}, v^* \rangle - \langle a_n, a_n^* \rangle + \varepsilon_{n}.
\]
Using the assumption that $F_{\partial f}(x,v^*) = \Phi_{f}(x,v^*)$ and that $(a_{n},a^*_{n})\in \mathrm{Gr}(\partial f)$, we may rewrite the above as
\begin{equation}\label{eq:4.8.1}
f(x)+f^*(v^*)+f(a_{n})+f^*(a^*_{n}) \leq \langle x, a_n^* \rangle + \langle a_{n}, v^* \rangle + \varepsilon_{n}.
\end{equation}
The above by the Fenchel-Young inequality \eqref{eq:4.8.1} yields
\[
f(a_{n})+f^*(v^*) \leq \langle a_n,v^* \rangle + \varepsilon_n
\]
and thus $a_{n}\in \partial_{\varepsilon_n}f^*(v^*)$. By Br\o{}nsted-Rockafellar (\cite[Theorem~3.17]{Phelps1993}) we may find a sequence $\{b_n, b_n^*\}_{n\in \mathbb{N}} \subseteq \mathrm{Gr}(\partial f)$ with
\begin{equation}\label{eq:4.8.2}
\max \left\{\|a_{n} - b_n\|,\,\,\|v^* - b_n^*\| \right\}\leq \sqrt{\varepsilon_n}.
\end{equation}
Therefore $v^* = \lim_{n \rightarrow \infty } b^*_{n}$, in norm. Since $\partial f^*$ is norm-to-norm upper semicontinuous at $(v^*,Df^*(v^*))= (v^*,\widehat{a})$ we get $b_{n} \xrightarrow{\|.\|} \widehat{a}$ and consequently by~\eqref{eq:4.8.2} 
\[a_{n}\xrightarrow[]{\|.\|}\widehat{a}.\]
Furthermore, inequality~\eqref{eq:4.8.1} also yields 
\begin{equation}\label{eq:4.8.3}
f(x)+f^*(a^*_{n}) \leq \langle x, a_{n}^* \rangle + \varepsilon_n
\end{equation}
so $a_n^* \in \partial_{\varepsilon_n} f(x)$. As $x \in \mathrm{int} (\mathrm{dom}\, f)$, there exists $M>0$ such that $\{a_{n}\}_{n \in \mathbb{N}} \in B(0,M)$. Let $a^*\in B(0,M)$ be a $w^*$-cluster point of $\{a^*_n\}_{n \in \mathbb{N}}$. Then 

\[(\widehat{a},a^*) \in \overline{\{(a_n,a^*_{n})\}_{n \geq 1}}^{(\|.\|,\,w^*)} \subseteq\mathrm{Gr}(\partial f), \] 
because as $\mathrm{int}(\mathrm{dom}\, f) \neq \varnothing$, the graph of $\partial f$ is $(\|.\|,\, w^*)$-closed and consequently $a^* \in \partial f(\widehat{a})$. Moreover, since for every $n\in \mathbb{N}$
\[
F_{\partial f}(x,v^*) \leq \langle x,a^*_{n} \rangle + \langle a_{n},v^* \rangle-\langle a_n,a^*_{n} \rangle + \e_{n}
\]
we deduce that $(\widehat{a},a^*) \in \mathcal{M}_{\partial f}(x,v^*)$ and clearly $\mathcal{M}_{\partial f}(x,v^*) \neq \varnothing$, which proves the first assertion of the lemma. \smallskip \\
(ii). It follows from the previous proposition and the fact that $\partial f^*(v^*)= \{ Df^*(v^*)\}$.
\end{proof}

Before proceeding, let us recall that if $v^*$ is a Fr\'{e}chet point of $f^*$ then
\[Df^*(v^*) \in \mathrm{dom}\, f \subseteq X\]
Let us also denote by $\pi_{X^*}:X \times X^* \rightarrow X^*$ the canonical projection onto $X^*$. For convenience, given a proper convex lsc function $f:X \rightarrow \mathbb{R} \cup \{+ \infty \}$ we define the sets 
\begin{equation}\label{eq:Definition of K}
\mathcal{Z} := \mathrm{Im}(Df^*) \subseteq \mathrm{dom}\, f
\qquad \text{and} \qquad
K := \overline{\mathrm{conv}}(\mathcal{Z}).
\end{equation}
The previous results yield the following corollary in case $\partial f$ has a unique representative, (i.e. $\mathcal{F}_{\partial f}$ is a singleton and Lemma~\ref{lem:singleton characterization} applies).

\begin{Cor}\label{Prop:Sub containment}
Let $f:X \rightarrow \mathbb{R} \cup \{+ \infty \}$ be proper, convex and lsc. Assume that $\mathcal{F}_{\partial f}$ is a singleton and ${\mathrm{dom}\, (\, Df^*) \neq \varnothing}$. Then for every $x \in \mathrm{int}(\mathrm{dom}\, f)$ and for every $\widehat{a} \in \mathcal{Z} := \mathrm{Im}(Df^*)$:
\[\partial f(x) \cap \partial f(\widehat{a}) \neq \varnothing.\]
In particular:
\begin{equation}\label{eq:Definition of V}
\mathcal{V}:= \overline{\mathrm{conv}}^{w^*}\left( \bigcup_{x \in \mathrm{int}(\mathrm{dom}\, f)} \partial f(x)\right) \subseteq \bigcap_{\widehat{x}\in \mathrm{conv(\mathcal{Z})}}\partial f(\widehat{x})
\end{equation}

\end{Cor}

\begin{proof}
Let $x \in \mathrm{int} (\mathrm{dom}\,f)$ and fix $\widehat{a} \in \mathcal{Z} \subseteq \mathrm{dom}\, f$. Take $\widetilde{x} \in \mathrm{int}(\mathrm{dom}\,f)$ such that
\begin{equation}\label{eq:4.9.1}
    x \in \left[ \widehat{a},\widetilde{x}. \right]
\end{equation}
By Lemma~\ref{Lem:Fitzpatrick maximizer}, $\mathcal{M}_{\partial f}(x,v^*) \neq \varnothing$, where $v^*$ is such that $Df^*(v^*)=\widehat{a}$ and thus
\begin{equation}\label{eq:4.9.2}
\partial f(\widetilde{x}) \cap \partial f(\widehat{a}) \supseteq \pi_{X^*}(\mathcal{M}_{\partial f}(x,v^*)) \neq \varnothing. 
\end{equation}
It follows from~\eqref{eq:4.9.1} and~\eqref{eq:4.9.2} that 
\[\partial f(x)=\partial f(\widetilde{x}) \cap \partial f(\widehat{a}) \subseteq \partial f(\widehat{a}).\]
Therefore
\[\left( \bigcup_{x \in \mathrm{int}(\mathrm{dom}\, f)} \partial f(x) \right) \subseteq \partial f(\widehat{a}), \,\,\, \text{for every} \,\,\,\widehat{a} \in \mathcal{Z}.\]
Finally, since $\partial f(\widehat{a})$ is $w^*$-closed and convex, it follows that $\mathcal{V} \subseteq \partial f(\widehat{a})$ for every $\widehat{a} \, \in \mathcal{Z}$. Furthermore, as $\mathcal{Z} \subseteq \mathrm{dom}\,f$ and $\mathrm{dom}\,f$ is convex we deduce $\mathrm{conv}(\mathcal{Z}) \subseteq \mathrm{dom}\,f$. Now, consider $\widehat{x} \in \mathrm{conv}(\mathcal{Z})$, that is
\[\widehat{x} = \sum_{i=1}^{N}\lambda_{i}\widehat{a}_{i}, \, \, \lambda_{i} \geq 0, \,\, \sum_{i=1}^{N} \lambda_{i} =1, \widehat{a}_{i} \in \mathcal{Z}.\]
Since $\bigcap_{i=1}^{n}\partial f(\widehat{a}_{i}) \supseteq \mathcal{V} \neq \varnothing$ by \cite[Proposition 22.10]{BauschkeCombetes} it follows that:
\[ f(\widehat{x}) = \bigcap_{i=1}^{n}\partial f(\widehat{a}_{i}).\]
Since $\widehat{x}$ was arbitrarily chosen we conclude that
\[\mathcal{V} \subseteq \bigcap_{\widehat{x} \in \mathrm{conv}(\mathcal{Z})} \partial f(\widehat{x}).\]
This proves the assertion.
\end{proof}
\noindent Recalling the notion from \eqref{eq:Definition of K} we have the following formula describing $f$ at interior points.

\begin{Prop}\label{Prop: values of f}
Let $f:X \rightarrow \mathbb{R} \cup \{+\infty\}$ be proper convex and lsc function such that $\mathcal{F}_{\partial f}$ is a singleton. Assume that $0 \in \mathcal{V}$ and $\mathrm{dom}\,(Df^*) \neq \varnothing$. Then
\[K \subseteq \mathrm{argmin}\, f \subseteq \mathrm{dom}\,f\qquad \text{and} \qquad \mathcal{V} \perp K-K.\] 
Moreover, for every $x \in \mathrm{int}(\mathrm{dom}\, f)$, $x^* \in \partial f(x)$ and $\widehat{x} \in K$
\[f(x)= \langle x-\widehat{x},x^* \rangle + \mathrm{min}f.\]

\end{Prop}

\begin{proof}
By Corollary~\ref{Prop:Sub containment} we have $0 \in \partial f(\widehat{x})$ for all $\widehat{x} \in \mathcal{Z}$, thus $\mathcal{Z} \subseteq \mathrm{argmin}\, f $. Since $\mathrm{argmin}\, f$ is closed and convex, we deduce 
\[\mathrm{conv}(\mathcal{Z}) \subseteq \mathrm{argmin} \,f \]
From \eqref{eq:Definition of V} and the fact that $\mathrm{Gr}(\partial f)$ is $(\|.\|,\,w^*)$-closed, 
\[ \mathcal{V} \subseteq \bigcap_{\widehat{x}\in K} \partial f(\widehat{x}).\]
Since $K \subseteq \mathrm{argmin}f$, we have $f(\widehat{x})=\mathrm{min}f$, for all $\widehat{x} \in K$. Let us now fix an arbitrary $\widehat{x} \in K$ and pick any $x\in \mathrm{int}(\mathrm{dom}\, f)$ and $x^* \in \partial f(x) \neq \varnothing$. Then by the above inclusion we also have $x^* \in \partial f(\widehat{x})$ and consequently
\begin{equation}
f(\widehat{x})+f^*(x^*)=\langle \widehat{x},x^* \rangle \quad \text{and} \quad f^*(x)+f^*(x^*)=\langle x,x^* \rangle.
\end{equation}
Combining the above we deduce:
\[f(x)= \langle x,x^*\rangle -f^*(x^*) = \langle x,x^*\rangle -\langle\widehat{x},x^* \rangle + f(\widehat{x})= \langle x-\widehat{x},x^*\rangle + \mathrm{min} f\]
and the formula for $f$ follows. Now, let $\widehat{x},\widehat{y} \in K$. We deduce from the above that for every $x \in \mathrm{int}(\mathrm{dom}\,f)$ and $x^* \in \partial f(x)$ we have
\[f(x)=\langle x-\widehat{x},x^* \rangle+ \mathrm{min}f = \langle x-\widehat{y},x^* \rangle + \mathrm{min}f\]
deducing
\[\langle \widehat{x}-\widehat{y},x^* \rangle =0.\]
Hence $\langle \widehat{x}-\widehat{y}, x^* \rangle = 0$ for all $\widehat{x}, \widehat{y} \in K$ and $x^* \in \bigcup_{x \in \mathrm{int}(\mathrm{dom}\, f)} \partial f(x)$, which yields $\mathcal{V} \perp K-K$.
\end{proof}

We shall also need the following lemma.

\begin{Lem}\label{Lem:Description of subK}
Let $f:X \rightarrow \mathbb{R}\cup \{+\infty\}$ be a proper lsc convex function. Assume $\mathrm{int}(\mathrm{dom}\, f) \neq \varnothing$, $\mathrm{dom}\,(Df^*) \neq \varnothing$ and that $F_{\partial f} = \Phi_{f}$ on $\mathrm{dom}\,f\times \mathrm{dom}\,f^*= \mathrm{dom}\,\Phi_f$. Then:
\[\mathrm{dom}\,(Df^*)\subseteq  \partial f(K)=  \mathcal{V}+ \widehat{N} \subseteq \mathrm{dom}\, f^*\]
where:
\begin{equation}\label{eq:Definition of N}
\widehat{N} = \bigcup_{\widehat{x}\in K} N_{\mathrm{dom}\,f}(\widehat{x}).
\end{equation}
and $K,\mathcal{V}$ are defined by~\eqref{eq:Definition of K} and~\eqref{eq:Definition of V} respectively.
\end{Lem}
\begin{proof}It follows readily from \eqref{eq:Definition of K} that 
\(\mathrm{dom}\,(Df^*) \subseteq \partial f(K)\).
By Corollary~\ref{Prop:Sub containment} for every $ \widehat{x} \in K$, we have $\mathcal{V} \subseteq \partial f(\widehat{x})$. Since $\partial f(\widehat{x})=\partial f(\widehat{x})+N_{\mathrm{dom}\,f}(\widehat{x})$  we deduce that
\(\partial f(\widehat{x}) \supseteq \mathcal{V}+N_{\mathrm{dom}\,f}(\widehat{x})\).
By Fact~\ref{BorweinYao2013} we obtain that 
\(\partial f(\widehat{x}) \subseteq \mathcal{V}+N_{\mathrm{dom}\,f}(\widehat{x})\)
and consequently
\[\partial f(\widehat{x})= \mathcal{V}+N_{\mathrm{dom}\,f}(\widehat{x})\]
As $\widehat{x}\in K$ is arbitrary, we conclude that:
\[\partial f(K)= \bigcup_{\widehat{x} \in K}\partial f(\widehat{x}) = \mathcal{V}+\bigcup_{\widehat{x} \in K}N_{\mathrm{dom}\,f}(\widehat{x}) = \mathcal{V}+\widehat{N}\]
This completes the proof.
\end{proof}
\noindent With this in hand, we are ready to provide a formula for $f^*$ in terms of $K,\mathcal{V},\widehat{N}$.
\begin{Prop}\label{Prop:typos f^*}
Let $f:X \rightarrow \mathbb{R}\cup \{+\infty \}$ be proper convex and lsc. Assume that $\mathrm{dom}\,(Df^*)$ is densely contained in $\mathrm{dom}\,f^*$, $\mathrm{int}(\mathrm{dom}\, f) \cap \mathrm{argmin}f \neq \varnothing$ and $\mathcal{F}_A$ is a singleton. Then for every $x^* \in X^*$ we have :
\begin{equation}\label{eq:typos f*}
    f^*(x^*)=\sigma_{K}(x^*)+i_{\overline{\mathcal{V}+N}}(x^*)-\mathrm{min}\, f
\end{equation}
\medskip
where $N=\overline{\mathrm{conv}} \widehat{N}$ and $K,\mathcal{V},\widehat{N}$ are defined by \eqref{eq:Definition of K}, \eqref{eq:Definition of V} and \eqref{eq:Definition of N} respectively.
\end{Prop}
\begin{proof} Define $g:X^* \to \mathbb{R}\cup \{+\infty \}$ by
\[g(x^*)=\sigma_{K}(x^*)+i_{\overline{\mathcal{V}+N}}(x^*)-\mathrm{min}f\]
Clearly $g$ is proper, convex and lsc. Since of $\mathrm{dom}\, f^*$ is convex Lemma~\ref{Lem:Description of subK} gives
\[\mathrm{dom}\,(Df^*) \subseteq \mathcal{V}+\mathrm{conv}\widehat{N} \subseteq \mathrm{dom}\,f^*\]
Therefore
\[\mathrm{dom}\, f^* \subseteq \overline{\mathrm{dom}\,}(Df^*) = \overline{\mathcal{V}+N} \]
which yields
\begin{equation}\label{eq:4.6.1}
i_{\mathrm{dom}\,f^*}(x^*) \geq i_{\overline{V+N}}(x^*), \,\,\, \text{for all} \,\,x^* \in X^*
\end{equation}
 Let $\overline{x} \in \mathrm{int}(\mathrm{dom}\, f)$ with $0 \in \partial f(\overline{x}) \subseteq \mathcal{V}$, i.e. $f(\overline{x})=\mathrm{min}f$ . Then since $0 \in \mathcal{V} \subseteq \partial f(\widehat{x})$ for all $\widehat{x} \in K$, (cf. Corollary~\ref{Prop:Sub containment}) given $x^* \in X^*$ we have
\begin{align*}
f(\overline{x})+f^*(x^*)=\Phi_{f}(\overline{x},x^*)=F_{\partial f}(\overline{x},x^*) &= \sup_{(a,a^*) \in \mathrm{Gr}(\partial f)} \left\{ \langle a,x^* \rangle +\langle \overline{x},a^* \rangle -\langle a,a^* \rangle \right\} \\[4pt]
&\geq \sup_{\widehat{x} \in K} \left\{ \langle \widehat{x},x^* \rangle +\langle \overline{x},0 \rangle- \langle \widehat{x},0\rangle\right\} \,=\,  \sigma_{K}(x^*).
\end{align*}
Therefore
\[f^*(x^*) \geq \sigma_{K}(x^*)-\mathrm{min}\, f\,. \]
This together with \eqref{eq:4.6.1} implies that for every $ x^* \in X^*$
\[f^*(x^*)= f^*(x^*)+i_{\mathrm{dom}\, f^*}(x^*) \geq\sigma_{K}(x^*)+i_{\overline{\mathcal{V}+N}}(x^*)-\mathrm{min} \,f = g(x^*)\]
Furthermore, for every $ v^* \in \mathrm{dom}\,(Df^*)$, by Proposition~\ref{Prop: values of f} we have 
$$\widehat{a}=Df^*(v^*) \in \mathrm{argmin}\,f.$$ 
In particular
\[f^*(x^*)+\mathrm{min}\,f=f^*(v^*)+f(\widehat{a}) = \langle \widehat{a},v^* \rangle \leq \sigma_{K}(v^*).\]
The above yields
\[f^*(v^*)=g(v^*),\quad \text{for every} \,\, v^* \in \mathrm{dom}\,(Df^*).\]
Since $\mathrm{dom}\,(Df^*)$ is dense in $\mathrm{dom}\,g$, we apply Fact~\ref{BD2002} to conclude 
$f^*(x^*)=g(x^*)$
for all $x^* \in X^*$, completing the proof.
\end{proof}
\noindent Having obtained an explicit formula for $f^*$, we now compute $f$.
\begin{Prop}\label{Prop:typos f}
    Let $f:X \rightarrow \mathbb{R}\cup \{+\infty \}$ be proper convex and lsc. Assume that $\mathrm{dom}\,(Df^*)$ is densely contained in $\mathrm{dom}\,f^*$, $\mathrm{int}(\mathrm{dom}\, f) \cap \mathrm{argmin}f \neq \varnothing$ and $\mathcal{F}_A$ is a singleton. Let $C=N^{\circ}=(\overline{\mathrm{conv}} \widehat{N})^{\circ}$ and $K,\mathcal{V}, \widehat{N}$ be defined as in \eqref{eq:Definition of K}, \eqref{eq:Definition of V}, \eqref{eq:Definition of N} respectively. Then for every $x \in X$ and $\widehat{x} \in K$ we have:
\begin{equation}\label{eq:typos f}
    f(x)=\sigma_{\mathcal{V}}(x-\widehat{x})+i_{\overline{K+C}}(x)+\mathrm{min}f
\end{equation}
\end{Prop}

\begin{proof}
Fix $\widehat{x}_{K} \in K$ and define a function $g$ as follows:
    \[g(x)=\sigma_{\mathcal{V}}(x-\widehat{x}_{K})+i_{\overline{K+N^{\circ}}}(x)+\mathrm{min}f.\]
Recall that $\mathcal{V} \perp K-K$ by Proposition~\ref{Prop: values of f}, thus for all $v^* \in \mathcal{V}$ and $\widehat{x}\in K$ 
\[\langle \widehat{x},v^* \rangle = \langle \widehat{x}_{K},v^* \rangle , \,\, \forall \widehat{x} \in K.\]
This implies that
\[\sigma_{\mathcal{V}}(x-\widehat{x})=\sigma_{\mathcal{V}}(x-\widehat{x}_{K})\,,\]
therefore the function $g$ is well defined and is independent of the choice of $\widehat{x}$. Note that
\[(i_{\overline{K+N^{\circ}}})^*(x^*)=\sigma_{\overline{K+N^{\circ}}}(x^*)=\sigma_{K+N^{\circ}}(x^*)=\sigma_{K}(x^*)+\sigma_{N^{\circ}}(x^*)=\sigma_{K}(x^*)+i_{N}(x^*),\]
where we used the fact that $\sigma_{N^o} \equiv i_{N}$ for the closed convex cone $N= \overline{\mathrm{conv}}N$. Recalling \eqref{eq:duality infimal conv}, we infer
\begin{equation}\label{eq:4.7.1}
    \notag g^*(x^*)+\mathrm{min}f=\overline{(\sigma_{\mathcal{V}}^{*}+\langle \widehat{x}_K,\cdot \rangle ) \square (i_{\overline{K+N^{o}}})^*}(x^*)=(\overline{i_{\mathcal{V}}+ \langle \widehat{x}_K,\cdot \rangle) \square (\sigma_{K}+i_{N})}(x^*).
\end{equation}
Furthermore:
\begin{align}\label{eq:4.7.2}
    \notag (i_{\mathcal{V}}+\langle \widehat{x}_{K}, \cdot \rangle ) \square (\sigma_{K}+i_{N})(x^*) &= \inf_{u^*+v^*=x^*} \left\{ i_{\mathcal{V}}(u^*)+\langle \widehat{x}_K,u^*\rangle +\sigma_{K}(v^*)+i_{N}(v^*)\right\} \\[6pt]
    &=\inf_{u^* \in \mathcal{V}} \left\{ \sigma_{K}(x^*-u^*)+\langle \widehat{x}_K , u^* \rangle +i_{N}(x^*-u^*)\right\}.
\end{align}
Since $\mathcal{V} \perp K-K$, we have again that for all $u^* \in \mathcal{V}$ and $\widehat{x}\in K$ we have:
\[\langle \widehat{x},u^* \rangle = \langle \widehat{x}_{K},u^* \rangle.\]
Therefore,
\[\sigma_{K}(x^*-u^*)=\sup_{\widehat{x} \in K} \{\langle \widehat{x},x^*-u^*\rangle\} = \sup_{\widehat{x}\in K} \langle \widehat{x},x^*\rangle - \langle \widehat{x}_{K},u^*\rangle = \sigma_{K}(x^*)-\langle \widehat{x}_{K},u^*\rangle \,.\]
Therefore the infimum in \eqref{eq:4.7.1} dissapears:
\[\inf_{u^* \in \mathcal{V}} \left\{ \sigma_{K}(x^*-u^*)+\langle \widehat{x}_K , u^* \rangle +i_{N}(x^*-u^*)\right\}= \sigma_{K}(x^*)+\inf_{u^*\in \mathcal{V}}i_{N}(x^*-u^*)=\sigma_{K}(x^*)+i_{N+K}(x^*)\]
and we obtain
\begin{equation}\label{eq:4.7.3}
g^*(x^*)+\mathrm{min}f=\overline{\sigma_{K}+i_{\mathcal{V}+N}}(x^*).
\end{equation}
By \eqref{eq:typos f*} $\sigma_{K}(x^*)+i_{\overline{\mathcal{V}+N}}:=f^*(x^*)+\mathrm{min}\,f$. Since this function is proper convex and lsc, it follows that $\overline{\sigma_{K}+i_{\mathcal{V}+N}}(x^*) \geq \sigma_{K}(x^*)+i_{\overline{\mathcal{V}+N}}(x^*)$ and consequently \eqref{eq:4.7.3} yields
\[g^*(x^*)+\mathrm{min}f \geq \sigma_{K}(x^*)+i_{\overline{\mathcal{V}+N}}(x^*)=f^*(x^*)+\mathrm{min}f\]
that is,
\[g^*(x^*)\geq f^*(x^*), \quad \text{for all} \quad x^* \in X^*\]
Moreover
\[g^{*}(x^*)=f^{*}(x^*), \quad \text{for all} \quad x^* \in \mathcal{V}+N\]
Since $\mathcal{V}+N$ is dense in $\mathrm{dom}\,f^*$ and $\mathrm{int}(\mathrm{dom}\, f^*) \neq \varnothing$ we deduce by Fact~\ref{BD2002} that
\[g^*(x^*)=f^{*}(x^*), \,\, \forall x^* \in X^*.\]
As $f,g$ are proper convex and lsc with $f^*=g^*$, it follows 
\[f(x)=g(x)=\sigma_{\mathcal{V}}(x-\widehat{x}_K)+i_{\overline{K+N}}(x)+\mathrm{min}f.\]
The proof is complete.
\end{proof}
\noindent As an immediate corollary, we deduce the following result, which corresponds to the \textit{only if} part of the statement of Theorem~B.
\begin{Cor}\label{mainCor}
Let $X$ be an RN space, $f : X \to \mathbb{R}\cup \{+\infty \}$ be proper, convex, lsc with $\mathrm{int}(\mathrm{dom}\,f) \neq \varnothing$ and 
$\mathrm{int}(\mathrm{dom}\,f^{*}) \neq \varnothing$. Then if $\mathcal{F}_{\partial f}$ is a singleton, there exist a constant $c \in \mathbb{R}$, a functional $\overline{x}^* \in X^*$ and closed convex sets $K,C \subseteq X$ and $\mathcal{V} \subseteq X^*$ where
\begin{itemize}
    \item  $C$ is a cone, $\mathcal{V}$ is $w^*$-closed 
    \item $0 \in \mathcal{V} \perp K-K.$
\end{itemize}
such that for every $\widehat{x} \in K$
\[
    f(x) \;=\; \sigma_{\mathcal{V}}(x - \widehat{x}) 
    \;+\; i_{\overline{K + C}}(x) 
    \;+\; \langle x, \overline{x}^* \rangle 
    \;+\; c \,.
\]
\end{Cor}
\smallskip
\begin{proof} Let us first notice that $f^*$ is a proper convex and $w^*$-lsc function. Since $X$ is an RNP-space, the dual space $X^*$ is $w^*$-Asplund and consequently the assumption $\mathrm{int}(\mathrm{dom}\,f^*) \neq \varnothing$ yields
$\overline{\mathrm{dom}\,(Df^*)} = \overline{\mathrm{dom}\,f^*}$. \smallskip\newline
Let further $\overline{x}\in \mathrm{int}(\mathrm{dom}\, f)$, $\overline{x}^* \in \partial f(\overline{x})$ and define the function
\[\widetilde{f}(x)\, := f(x)-\langle x-\overline{x},\overline{x}^* \rangle, \quad x\in X. \]
Since $\mathcal{F}_{\partial f} = \{ F_{\partial f} \}$, we obtain by translation invariance (Proposition~\ref{prop:Translation Invariance}) that $\mathcal{F}_{\partial \widetilde{f}} = \{F_{\partial \widetilde{f}}\}$. In addition, 
$$\overline{x} \in \mathrm{int}(\mathrm{dom}\,\widetilde{f}) \bigcap \mathrm{arg}\mathrm{min}\widetilde{f}\qquad\text{and}\qquad \widetilde{f}(\overline{x})=f(\overline{x})=\mathrm{min}\widetilde{f}.$$ 
The assumptions of Proposition~\ref{Prop:typos f} are thus satisfied for $\widetilde{f}$, allowing us to conclude that for every $x\in X$
\[f(x)=\sigma_{\mathcal{V}}(x)+i_{\overline{K+C}}(x)+\langle x,\overline{x}^* \rangle-\langle \overline{x},\overline{x}^*\rangle+f(\overline{x}) \]
where the sets $\mathcal{V},C$ and $K$ are defined as before. Setting $c=f(\overline{x})-\langle \overline{x},\overline{x}^* \rangle$ the assertion follows.
\end{proof}

\subsection{A class of subdifferentials with unique representative}
The aim of this section is to prove that the class of functions introduced in Section~4.1, satisfy $\mathcal{F}_{\partial f}=\{F_{\partial f}\}$. Before proceeding, we introduce a notation. For a proper convex lsc function $f:X \rightarrow \mathbb{R}\cup \{+\infty \}$, we denote
\[(\,f^* \oplus f\,)(x,x^*)=(f^*(x^*),f(x)) \in \mathbb{R}^2\]
and define
\[\partial_{\e}(\,f^*\oplus f\,)(x,x^*)= \left\{ (y,y^*) \in X \times X^* \colon \,\,  y^* \in \partial_{a}f(x), \,\, y \in \partial_{b} f^*(x^*),  \,\, a+b \leq \e \right\}.\]
Consider further the function $K_{\partial f}: X \times X^* \rightarrow \mathbb{R}\cup \{+\infty \}$ defined via
\begin{equation}\label{eq:Def of K_f}
K_{\partial f}(x,x^*) := \inf\{\varepsilon\ge 0 : \operatorname{Gr}(\partial f)\bigcap \partial_\varepsilon(\,f^*\oplus f\,)(x,x^*)\neq\varnothing \}\,.
\end{equation}
We are now ready to proceed with the main lemma of this section.

\begin{Lem}\label{Thm:F_A from f+f^*}
Let $f:X\rightarrow \mathbb{R}\cup \{+\infty \}$ be proper, convex and lsc. For $(w,v^*) \in \mathrm{dom}\,f \times \mathrm{dom}\,f^*$

\[
\Phi_{f}(w,v^*)=f(w)+f^*(v^*) = F_{\partial f}(w,v^*) + K_{\partial f}(w,v^*).
\]
\end{Lem}

\begin{proof}
Fix $(w,v^*) \in \mathrm{dom}\,f \times \mathrm{dom}\,f^*$ and notice that since $F_{\partial f}$ and $\Phi_{f}$ belong to $\mathcal{F}_{\partial f}$ and the former is minimal, we have
\[F_{\partial f}(w,v^*) \leq f(w)+f^*(v^*) < + \infty.\] 
\medskip
For any arbitrary $\eta > 0$ , there exist $a,b \geq 0$ and $(y,y^*) \in \mathrm{Gr}(\partial f)$ with $y \in \partial_a f^*(v^*)$ and $y^* \in \partial_b f(w)$ such that $K_{\partial f}(w,v^*)+\eta \geq a+b$. This yields
\[
f(w)+f^*(y^*) \leq \langle w, y^* \rangle + a \quad \text{and} \quad f(y)+f^*(v^*) \leq \langle y, v^* \rangle + b.
\]
By \eqref{eq:Definition of Fitzpatrick2} and the fact that $f(y)+f^*(y^*)=\langle y,y^* \rangle$ we have 
\begin{align*}
F_{\partial f}(w,v^*) &\geq \langle w, y^* \rangle + \langle y, v^* \rangle - \langle y, y^* \rangle \\
&\geq f(w)+f^*(y^*) + f(y)+f^*(v^*) - f(y)-f^*(y^*) - a - b \\
&= f(w) + f^*(v^*) - a - b \geq f(w)+f^*(v^*) - K_{\partial f}(w,v^*) - \eta.
\end{align*}
Thus
\begin{equation}\label{eq:L4.2.1}
K_{\partial f}(w,v^*)+F_{\partial f}(w,v^*) \geq f(w)+f^*(v^*).
\end{equation}
For the reverse inequality, fix $\e >0$ and let $(z,z^*) \in \mathrm{Gr}(\partial f)$ be such that
\[
F_{\partial f}(w,v^*) \leq \langle w,z^* \rangle + \langle z,v^* \rangle - \langle z,z^* \rangle + \varepsilon.
\]
Then
\begin{align*}
f(w)+f^*(v^*)-F_{\partial f}(w,v^*) &\geq f(w)+f^*(v^*)-\langle w,z^*\rangle -\langle z,v^* \rangle + \langle z,z^* \rangle - \e \\
&=
\underbrace{ f(w)+f^*(z^*)-\langle x,z^* \rangle }_{a}+\underbrace{f(z)+f^*(v^*)-\langle z,v^* \rangle }_{b}-\e \\
& \geq K_{\partial f}(w,v^*)- \e
\end{align*}
which together with \eqref{eq:L4.2.1} concludes the proof.
\end{proof}

\begin{Rem}\normalfont
Let us extract from the proof of Lemma~\ref{Thm:F_A from f+f^*} the following result: given a point $(x,x^*) \in \mathrm{dom}\,f \times \mathrm{dom}\,f^*$, the infimum in~\eqref{eq:Def of K_f} is actually attained at some $\e \geq 0$ if and only if the supremum in the definition of the Fitzpatrick function~\eqref{eq:Definition of Fitzpatrick2} is attained.
\medskip
\end{Rem}

\noindent An immediate consequence is the following result

\begin{Cor}\label{Cor:Singleton Criterion}
Let $f: X \to \mathbb{R}\cup \{+\infty \}$ be proper, convex, and lsc. Then $\mathcal{F}_{\partial f} = \{F_{\partial f}\}$ if and only if
\[ \mathrm{dom}\,F_{\partial f} = \mathrm{dom}\,f \times \mathrm{dom}\,f^* \]
and for every $(x,x^*) \in \mathrm{dom}\,f \times \mathrm{dom}\,f^*$ and $\e >0$
\[
\partial_\varepsilon \left(f^* \oplus f \right)(x,x^*) \cap \mathrm{Gr}(\partial f) \neq \emptyset\,.
\]
\end{Cor}

\begin{proof}
The two conditions imply that for every $ (x,x^*) \in X \times X^*$
\[F_{\partial f}(x,x^*)= f(x)+f^*(x^*)=\Phi_{f}(x,x^*)\]
and the result follows from Lemma~\ref{lem:singleton characterization}(iv). \end{proof}

\noindent With this criterion in hand, we now generalize ~\cite[Theorem~5.3]{BBBRW2007} to a broader class of functions. 

\begin{Thm}\label{Thm:Main2}
Consider closed convex sets $K,C \subseteq X$ and $\mathcal{V} \subseteq X^*$ where
\begin{itemize}
    \item $C$ is a cone, $\mathcal{V}$ is $w^*$-closed
    \item $0 \in \mathcal{V} \perp K-K$
\end{itemize}
Assume further the following compatibility condition:
\begin{equation}\label{eq:compatibility condition}
\overline{(\sigma_{K}+i_{\mathcal{V}+C^{\circ}})}(x^*) = \sigma_{K}(x^*)+i_{\overline{\mathcal{V}+C^{\circ}}}(x^*).
\end{equation}
Then for every $\widehat{x}\in K$ define $f:X \rightarrow \mathbb{R}\cup \{+\infty \}$  by
\[
    f(x) \;=\; \sigma_{\mathcal{V}}(x - \widehat{x}) 
    \;+\; i_{\overline{K + C}}(x) \,, \,\, x \in X.
\]
Then the Fitzpatrick family of its subdifferential reduces to a singleton: $\mathcal{F}_{\partial f} = \{F_{\partial f}\}\,.$
\end{Thm}

\begin{proof}
Reasoning as in Proposition~\ref{Prop:typos f}, we deduce that  $f$ is well defined, independently of the choice of $\widehat{x} \in K$. Fix $\widehat{x}_K \in K$. Notice that since $\mathcal{V} \perp K-K$:
\[\mathrm{min} \, f=f(\widehat{x}) = 0 \quad\text{for all} \quad  \widehat{x} \in K\]
\medskip
\noindent \textit{Step 1.} We compute $f^*$ and show that $\mathcal{V} \subseteq \partial f(\widehat{x})$ for all $\widehat{x} \in K$. Similarly to the computation of Proposition~\ref{Prop:typos f} we have:
\begin{equation*}
f^*(x^*)=\overline{\sigma_{K}+i_{\mathcal{V}+C^{\circ}}}(x^*)
\end{equation*}
and using \eqref{eq:compatibility condition} we conclude that in fact 
\begin{equation}\label{eq:4.4.3}
    f^*(x^*)= \sigma_{K}(x^*)+i_{\overline{\mathcal{V}+C^{\circ}}}(x^*).
\end{equation}
Furthermore, for any $v^* \in \mathcal{V}$ we have
\[\langle \widehat{x},v^*\rangle = \langle \widehat{y},v^* \rangle, \,\, \text{for every } \,\, \widehat{x},\, \widehat{y} \in K, \]
which yields
\[f^*(v^*) = \sigma_{K}(v^*)=\langle \widehat{x},v^* \rangle. \]
Since $f(\widehat{x})=0$, for all $\widehat{x} \in K$, it follows that
\[f(\widehat{x})+f^*(v^*)=\langle \widehat{x},v^* \rangle, \quad \text{for every} \quad \widehat{x} \in K\]
therefore in particular $\mathcal{V} \subseteq \bigcap_{\widehat{x} \in K} \partial f(\widehat{x}_{K})$.

\medskip
\noindent \textit{Step 2}. We show that
\begin{equation}\label{eq:domains}
\mathrm{dom}\,F_{\partial f} = \mathrm{dom}\,f \times \mathrm{dom}\,f^*
\end{equation}
To this end, let us recall, see e.g. \cite[Theorem 2.6]{bauschke2006fitzpatrick}, that
\begin{equation}\label{eq:4.4.4}
\mathrm{dom}\,f \times \mathrm{dom}\,f^* \subseteq \mathrm{dom}\,F_{\partial f} \subseteq \overline{\mathrm{dom}\,}\,f \times \overline{\mathrm{dom}\,}\,f^*.
\end{equation}
If $w \in \overline{\mathrm{dom}\,}\,f \, \setminus \mathrm{dom}\,f$, then since
\[\overline{\mathrm{dom}\,}f \subseteq \overline{K+C}\]
we infer that
\[\sigma_{\mathcal{V}}(w-\widehat{x}_K)=+ \infty.\]
Consequently there exists a sequence $\{v^*_{n}\}_{n \in \mathbb{N}} \in \mathcal{V} \subseteq \partial f(\widehat{x}_{K})$ such that $\langle w-\widehat{x}_K,v^*_{n}\rangle \rightarrow + \infty$. It follows that for every $u^* \in X^*$,
\[F_{\partial f}(w,u^*) \geq \langle w,x^*_{n}\rangle +\langle 0,u^* \rangle - \langle \widehat{x}_K,x^*_{n}\rangle = \langle w-\widehat{x}_K,x^*_{n} \rangle \rightarrow + \infty.\]
As $u^*$ is arbitrary, we obtain
\[\mathrm{dom}\,F_{\partial f} \subseteq \mathrm{dom}\,f \times \overline{\mathrm{dom}\,}f^*.\]
On the other hand, by \eqref{eq:4.4.3}, 
$\overline{\mathrm{dom}\,}f^* \subseteq \overline{\mathcal{V}+C^{\circ}}$. Thus, given $y^* \in \overline{\mathrm{dom}\,}f^* \, \setminus \mathrm{dom}\,f^*$ we have
\[\sigma_{K}(y^*)= +\infty\,,\]
and there exists a sequence  $\{\widehat{x}_{n}\}_{n \in \mathbb{N}} \subseteq K$ such that $\langle \widehat{x}_{n},y^* \rangle \rightarrow +\infty$. Since $0 \in \mathcal{V} \subseteq \partial f(\widehat{x}_{n})$ for all $ n \in \mathbb{N}$, given $z \in X$ we have
\[F_{\partial f}(z,y^*) \geq \langle z,0\rangle + \langle \widehat{x}_{n},y^* \rangle -\langle \widehat{x}_{n},0 \rangle \rightarrow + \infty.\]
Therefore, $\mathrm{dom}\,F_{\partial f} \subseteq \mathrm{dom}\,f \times \mathrm{dom}\,f^*$ and by \eqref{eq:4.4.4} we deduce that \eqref{eq:domains} holds.

\medskip

\noindent \textit{Step 3.} We prove that for every $(x,x^*) \in \mathrm{dom}\,f \times \mathrm{dom}\,f^*$ and $\e >0$
\[\partial_{\e}(f^* \oplus f)(x,x^*) \cap \mathrm{Gr}(\partial f) \neq \varnothing\]
To this end, fix $\e>0$. From \eqref{eq:4.4.3}, we may find $\widehat{y} \in K \subseteq \mathrm{argmin}\,f$, such that $\sigma_{K}(x^*) \leq \langle \widehat{y},x^* \rangle+ \e/2$. Since $f(\widehat{y})=0$, we obtain
\[\widehat{y} \in \partial_{\e /2} f^*(x^*).\]
We also have that $\partial_{\e /2} f(x) \cap \mathcal{V} \neq \varnothing$ for any $\e >0$. To see this, as $x \in \mathrm{dom}\, f$, take $v^* \in \mathcal{V}$ with
\[f(x)=\sigma_{\mathcal{V}}(x-\widehat{x}_K) \leq \langle x-\widehat{x}_K,v^* \rangle +\frac{\e}{2} = \langle x,v^* \rangle -\langle \widehat{x}_{K},v^* \rangle +\frac{\e}{2} = \langle x,v^* \rangle -f^*(v^*) +\frac{\e}{2}\]
where we used Step~1 and the fact that $f(\widehat{x}_{K})=0$. Now, since $\mathcal{V} \subseteq \partial f(\widehat{y})$ clearly
\[\partial_{\e /2}f(x) \cap \partial f (\widehat{y}) \neq \varnothing \]
For $y^* \in \partial_{\e /2}f(x) \cap \partial f (\widehat{y})$ we have: 
\[(\widehat{y},y^*) \in \partial_{\e}(f^* \oplus f)(x,x^*) \bigcap \mathrm{Gr}(\partial f) \neq \varnothing
\]
Thus, both conditions of Corollary~\ref{Cor:Singleton Criterion} are satisfied, concluding the proof.
\end{proof}

\begin{Rem}\normalfont
\medskip

(i). The compatibility condition~\eqref{eq:compatibility condition} is satisfied in several natural situations. In particular, any of the following assumptions (a)-(b) ensures~\eqref{eq:compatibility condition}: \smallskip\newline
(a) $K$ is compact \quad (b) $\mathcal{V}+C^{\circ}$ is closed \quad
(c) $\mathrm{int}(\,\mathrm{dom}\,f^*) \neq \varnothing$ \quad 
(d) $X$ is finite dimensional \smallskip\newline 
(The later case which will be discussed in the next subsection.)
\medskip

(ii). In the degenerate case where $K=\{0\}$ and $C=X$, we recover \cite[Theorem~5.3]{BBBRW2007}, namely the result for sublinear functions. In a dual manner, if $\mathcal{V}=\{0\}$ and $C=\{0\}$, we recover the case of indicator functions of closed convex sets. 
\medskip

(iii). Computing conjugates we see that both $f$ and $f^*$ belong to the same family of functions 
(up to natural adjustments). This is consistent with \cite[Theorem~5.7]{BBBRW2007}, which states that 
if the subdifferential of the function $f^*$ satisfies $\mathcal{F}_{\partial f^*}=\{F_{\partial f^*}\}$, then so does the subdifferential of $f$.
\end{Rem}
\noindent Combining Theorem~\ref{Thm:Main2} with Corollary~\ref{mainCor} we obtain Theorem~B. \medskip\newline
%
\noindent\rule{4cm}{1.6pt}\smallskip \newline
\noindent\textit{Proof of Theorem~B.} \smallskip\newline 
It is sufficient to verify the compatibility condition~\eqref{eq:compatibility condition}. A close inspection of the proof of Proposition~\ref{Prop:typos f}, shows that \eqref{eq:compatibility condition} holds whenever $\mathrm{int}(\mathrm{dom}\,f^*) \neq \varnothing$.\hfill$\square$
\bigskip\newline
\noindent Combining the above result with the main result of Section~3 yields the following.

\begin{Cor}\label{Combination}
Let $X$ be a space with the RNP, and $A: X \rightrightarrows X^*$ a 3-monotone and maximal monotone operator 
admitting only one representative function. If $\mathrm{int}(\mathrm{dom}\,(A))$ and $\mathrm{int}(\mathrm{Im}\,(A))$ are nonempty, 
then $A= \partial f$ for some proper convex lsc function $f$ of the form \eqref{eq:definition of f}.
\end{Cor}

\begin{Rem}\normalfont
It remains open to determine whether the conditions ${\mathrm{int}(\mathrm{dom}\,f)\neq\varnothing}$ and ${\mathrm{int}(\mathrm{dom}\,f^*) \neq \varnothing}$ as well as the assumption that the ambient space is an RNP space are necessary for Theorem~B (and consequently for Corollary~\ref{Combination}). 
\end{Rem}
\noindent On the other hand, the only non-subdifferential operators known to admit a unique representative function are the skew-symmetric linear operators. 
This motivates the following conjecture.
\begin{Conj}
A maximal monotone operator admitting a unique representation 
must either be the subdifferential of a function of the form~\eqref{eq:definition of f}
or a linear skew-symmetric operator.
\end{Conj}

\subsection{The Finite-Dimensional Case}

\noindent We now focus on the finite-dimensional setting, where the technical assumptions required in Theorem~\ref{Thm:Main2} can be entirelly removed. Indeed, in the aforementioned statements conditions on the Banach space $X$ and on the domains of $f$ and $f^*$ were needed solely to guarantee the existence of a maximizer in \eqref{eq:Definition of Fitzpatrick2}, i.e., to establish Lemma~\ref{Lem:Fitzpatrick maximizer}. In finite dimensions, however, such a maximizer always exists under minimal assumptions, even without requiring the Fitzpatrick family being singleton. Throughout this section, $E$ denotes a finite-dimensional Euclidean space.  

\medskip

Before proving the main result, we first recall some standard properties of convex functions in finite dimensions. Let $f \colon E \to \mathbb{R}\cup \{+\infty \}$ be a proper convex function. Its domain lies within the affine hull $\mathrm{Aff}(\mathrm{dom}\,f)$. The restriction of $f$ to this space has nonempty interior, denoted by $\operatorname{rint}(\mathrm{dom}\,f)$. Fix $u \in \mathrm{Aff}(\mathrm{dom}\,f)$ and define the associated linear subspace of $E$ given by 
\[U:=\mathrm{Aff}(\mathrm{dom}\,f)-\{u\}.\]
For each $x^* \in E^*$, we decompose
\[
x^* = x^*_{\mathrm{par}} + (x^*)^\perp,
\]
where $x^*_{\mathrm{par}}$ is the projection onto $U$ and $(x^*)^\perp \in U^\perp$. Consequently, for all $x,y \in \mathrm{Aff}(\mathrm{dom}\,f)$, we have
\begin{equation}\label{eq:orthogonality decomposition}
\langle x-y,(x^*)^{\perp}\rangle =0.
\end{equation}
\noindent For $\e \geq 0$ we define
\[
\partial_{\e}^{\mathrm{par}} f(x) := \{\, x^*_{\mathrm{par}} : x^* \in \partial_{\e} f(x)\,\},
\]
the projection onto $U$ of the set of $\e$-subgradients of $f$ at $x$. Restricting $f$ to its affine hull,
\[
\tilde{f} \colon= f|_{\mathrm{Aff}(\mathrm{dom}\,f)} \,\,\colon \,\,\mathrm{Aff}(\mathrm{dom}\,f) \to \mathbb{R}\cup \{+\infty \},
\]
we obtain $\operatorname{rint}(\mathrm{dom}\,f) = \operatorname{int}(\mathrm{dom}\,\tilde{f}) \neq \varnothing$. Hence for every $\e \geq 0$ and $x \in \operatorname{rint}(\mathrm{dom}\,f)$, the set $\partial_{\e}^{\mathrm{par}} f(x)$ is bounded. More precisely:
\begin{equation}\label{eq:local boundedness}
    \|y_{\mathrm{par}}^{*}\| \leq M, \quad \text{for every } \quad y^* \in \partial_{\e} f(x)
\end{equation}

Analogously, let $v \in \mathrm{Aff}(\mathrm{dom}\,f^*)$ and define 
\[V = \mathrm{Aff}(\mathrm{dom}\,f^*)-\{v\}.\] 
Any $y \in E$ admits the decomposition
\[
y = y_{\mathrm{par}} + y^\perp,
\]
with $y_{\mathrm{par}} \in V$ and $y^\perp \in V^\perp$. If $y \in \partial_{\e} f^*(x^*)$ for some $\e \geq 0$ and $x^* \in \operatorname{rint}(\mathrm{dom}\,f^*)$, then $y_{\mathrm{par}}$ is also bounded, similar to \eqref{eq:local boundedness}.  

Fixing $u \in \mathrm{Aff}(\mathrm{dom}\,f)$ and $v \in \mathrm{Aff}(\mathrm{dom}\,f^*)$, one checks that for any $(y,y^*) \in \mathrm{Gr}(\partial f)$, the projection $(y_{\mathrm{par}},y^*_{\mathrm{par}})$ also lies in $\mathrm{Gr}(\partial f)$. We therefore define the parallel projection of the graph with respect to $U$ and $V$:
\[
\mathrm{Gr}^{\mathrm{par}}(\partial f) := \left\{ (y_{\mathrm{par}},y^*_{\mathrm{par}}) : (y,y^*) \in \mathrm{Gr}(\partial f)\right\} \subseteq (U \times V) \bigcap \mathrm{Gr}(\partial f).
\]

We are now in position to show that the supremum in \eqref{eq:Definition of Fitzpatrick2} is always attained in finite dimensions whenever $(x,x^*) \in \operatorname{rint}(\mathrm{dom}\,f) \times \operatorname{rint}(\mathrm{dom}\,f^*)$:

\begin{Prop}\label{Thm:Finite Fitzpatrick maximizer}
Let $f \colon E \to \mathbb{R}\cup \{+\infty \}$ be a proper, convex and lsc on a finite-dimensional Euclidean space $E$. Suppose
\[
(x, x^*) \in \operatorname{rint}(\mathrm{dom}\,f) \times \operatorname{rint}(\mathrm{dom}\,f^*).
\]
Then 
\[\mathcal{M}_{\partial f}(x,x^*) \neq \varnothing.\]
where $\mathcal{M}_{\partial f}(x,x^*)$ is defined as in \eqref{eq:definition of M_A}.
\end{Prop}

\begin{proof}
Since $(x,x^*)$ lies in the product of the relative interiors of $f$ and $f^*$, $F_{\partial f},\Phi_{f} \in \mathcal{F}_{\partial f}$ and $F_{\partial f}$ is minimal 
\[
F_{\partial f}(x,x^*) \leq \Phi_{f}(x,x^*) = f(x) + f^*(x^*) < +\infty,
\]
so the supremum is finite. Fix $(u,v) \in \mathrm{Aff}(\mathrm{dom}\,f) \times \mathrm{Aff}(\mathrm{dom}\,f^*)$ and define $U,V$ as above. We first restrict the supremum to parallel components. Define
\[
F^{\mathrm{par}}_{\partial f}(x,x^*) := \sup_{(y_{\mathrm{par}},y^*_{\mathrm{par}}) \in \mathrm{Gr}^{\mathrm{par}}(\partial f)}
\big\{ \langle x, y^*_{\mathrm{par}} \rangle + \langle y_{\mathrm{par}}, x^* \rangle - \langle y_{\mathrm{par}}, y^*_{\mathrm{par}} \rangle \big\}.
\]
Clearly $F_{\partial f}(x,x^*) \geq F^{\mathrm{par}}_{\partial f}(x,x^*)$. Take $\e>0$ and $(y,y^*) \in \mathrm{Gr}(\partial f)$ such that
\begin{equation}\label{eq:finite_fitz_approx}
F_{\partial f}(x,x^*) \leq 
\langle x, y^* \rangle + \langle y, x^* \rangle - \langle y,y^* \rangle + \e.
\end{equation}
Writing $y = y_{\mathrm{par}} + y^\perp$ and $y^* = y^*_{\mathrm{par}} + (y^*)^\perp$, we infer from the orthogonality relation \eqref{eq:orthogonality decomposition} that
\[
\langle x-y, (y^*)^\perp \rangle = 0
\quad \text{and} \quad
\langle y^\perp, x^* - y^*_{\mathrm{par}} \rangle = 0.
\]
Substituting into~\eqref{eq:finite_fitz_approx}, we obtain
\[
F_{\partial f}(x,x^*) \leq 
\langle x, y^*_{\mathrm{par}} \rangle + \langle y_{\mathrm{par}}, x^* \rangle - \langle y_{\mathrm{par}}, y^*_{\mathrm{par}} \rangle + \e
\leq F^{\mathrm{par}}_{\partial f}(x,x^*) + \e.
\]
As $\e > 0$ is arbitrary, it follows that
\[
F_{\partial f}(x,x^*) = F^{\mathrm{par}}_{\partial f}(x,x^*).
\]
Finally, by Lemma~\ref{Thm:F_A from f+f^*}, the supremum may be restricted to points $(y_{\mathrm{par}},y^*_{\mathrm{par}})$ lying in 
\begin{equation}\label{eq:maximizing set}
\partial^{\mathrm{par}}_{\rho}(f^* \oplus f)(x,x^*) \cap \mathrm{Gr}(\partial f),
\end{equation}
for some $\rho > f(x)+f^*(x^*)-F_{\partial f}(x,x^*)$.  
Since $(x,x^*)$ lies in the respective relative interiors of $f,f^*$, by \eqref{eq:local boundedness} we have that $\partial^{\mathrm{par}}_{k}(f^* \oplus f)(x,x^*)$ is bounded. As both $\partial^{\mathrm{par}}_{k}(f^* \oplus f)(x,x^*)$ and $\mathrm{Gr}(\partial f)$ are closed, the intersection in \eqref{eq:maximizing set} is compact. This completes the proof, as the supremum of a continuous function over a compact set is attained. 
\end{proof}

\medskip

In view of Proposition~\ref{Thm:Finite Fitzpatrick maximizer}, all results from Section~4.1 extend to the finite-dimensional case, with no extra assumptions on $\mathrm{dom}\,f$ and $\mathrm{dom}\, f^*$, up to relative adjustments. Yielding the following full characterization:

\begin{Thm}\label{thm:main4}
Let ${f \colon E \to \mathbb{R}\cup \{+\infty \}}$ be proper, convex, lsc and $E$ be finite dimensional Euclidean space. Then the Fitzpatrick family, $\mathcal{F}_{\partial f}$, consists of a single element
if and only if there exist a constant $c \in \mathbb{R}$, a functional $\overline{x}^* \in E^*$ and closed convex sets $K,C, \mathcal{V} \subseteq E$ with $C$ a cone with
\[
    0 \in \mathcal{V} \perp K-K
\]
such that, for every $\widehat{x} \in K$
\[
    f(x) \;=\; \sigma_{\mathcal{V}}(x - \widehat{x}) 
    \;+\; i_{\overline{K + C}}(x) 
    \;+\; \langle x, \overline{x}^* \rangle 
    \;+\; c , \,\, \text{for}\,\, x \in E.
\]
\end{Thm}
\begin{proof}

The assertion follows by applying the arguments from Section~4.3 in the finite-dimensional context. In particular since $E$ is finite dimensional, Rademacher’s theorem ensures that $f^*$ is differentiable almost everywhere on $\operatorname{rint}(\mathrm{dom}\,f^*)$ and  Fitzpatrick supremum is attained in virtue of Proposition~\ref{Thm:Finite Fitzpatrick maximizer}, allowing us to reproduce the arguments of Lemma~\ref{Prop:Sub containment} and Corollary~\ref{mainCor} verbatim within~$E$. We need only check the compatibility condition of Theorem~\ref{Thm:Main2}:
\[f^*(x^*)=\overline{(\sigma_{K}+i_{\mathcal{V}+N})}(x^*) = \sigma_{K}(x^*)+i_{\overline{\mathcal{V}+N}}(x^*):= g(x^*)\]
The two convex lsc functions above, agree on $\mathcal{V}+N$ which is dense in $\mathrm{rint}(\mathrm{dom}\,g)$ and clearly as the lsc envelope of the sum is greater than the sum of the lsc envelopes, $f^* \geq g$ always. In virtue of Fact~\ref{BD2002} the conclusion follows. \end{proof}
\bigskip

\noindent \textbf{Acknowledgments} The authors would like to thank Pierre-Cyril Aubin-Frankowski, Jér\^ome Bolte, Heinz Bauschke, Guillaume Carlier, Juan Enrique Mart\'{i}nez-Legaz, Nicolas Hadjisavvas and Sebasti\'{a}n Tapia-Garc\'{i}a for useful discussion and suggestions. \smallskip\newline 
This research was funded in whole by the Austrian Science Fund (FWF) [DOI 10.55776/P36344N]. For open access purposes, the second author has applied a CC BY
public copyright license to any author accepted manuscript version arising from this submission.

\smallskip

\footnotesize
\bibliography{references}
\bibliographystyle{siam}

\bigskip

\noindent
\textbf{Sotiris Armeniakos, Aris Daniilidis} \smallskip\newline
Institut für Stochastik und Wirtschaftsmathematik, VADOR E105-04 \\
TU Wien, Wiedner Hauptstraße 8, 1040 Vienna, Austria \smallskip\newline
Email: \texttt{sotirios.armeniakos@tuwien.ac.at, aris.daniilidis@tuwien.ac.at} \\
Webpage: \href{https://www.arisdaniilidis.at}{\texttt{https://www.arisdaniilidis.at}} \smallskip\newline
Research supported by the Austrian Science Fund (FWF), grant DOI \href{https://doi.org/10.55776/P-36344N}{10.55776/P-36344N}.
\end{document}